\documentclass[pamq]{ipart}

\RequirePackage{hyperref}

\startlocaldefs

\newcommand\be{\begin{eqnarray}}
\newcommand\ee{\end{eqnarray}}

\newcommand{\R}{{\mathbb R}}

\theoremstyle{plain}

\newtheorem{proposition}{Proposition}[section]
\newtheorem{definition}{Definition}[section]

\endlocaldefs

\pubyear{2025}
\volume{21}
\issue{5}
\firstpage{1}
\lastpage{1}

\begin{document}

\begin{frontmatter}

\title{Dynamics of Cayley Forms}

\begin{aug}

\author{\fnms{Kirill} \snm{Krasnov}\ead[label=e1]{kirill.krasnov@nottingham.ac.uk, ORCID: 0000-0003-2800-3767}}
\address{School of Mathematical Sciences, University of Nottingham, Nottingham, NG7 2RD, UK\\\printead{e1}}
\end{aug}

\begin{abstract}\noindent The most natural first-order PDE's to be imposed on a Cayley 4-form in eight dimensions is the condition that it is closed. As is well-known, this implies  integrability of the ${\rm Spin}(7)$-structure defined by the Cayley form, as well as Ricci-flatness of the associated metric. In this work, we investigate the natural second-order conditions that can be imposed. We start at the linearised level, and construct the most general diffeomorphism-invariant second order in derivatives Lagrangian that is quadratic in the perturbations of the Cayley form, finding a two-parameter family. We then describe a non-linear completion of the linear story. We parametrise the intrinsic torsion of a ${\rm Spin}(7)$-structure by a 3-form, and show that this 3-form is completely determined by the exterior derivative of the Cayley form. The space of 3-forms splits into two ${\rm Spin}(7)$ irreducible components, and there is a two-parameter family of diffeomorphism-invariant Lagrangians quadratic in the torsion, matching the linearised story. We then describe a first-order in derivatives version of the action functional, which depends on the Cayley 4-form and auxiliary 3-form as independent variables. 

Our construction yields two distinguished natural Lagrangians. One of them is selected by the condition that the Euler-Lagrange equation for the auxiliary 3-form requires it to coincide with the torsion 3-form, leading to a canonical torsion-squared functional whose field equations we analyse. In the second, a specific linear combination of the two torsion-squared invariants is shown to integrate to the scalar curvature, and the resulting Euler–Lagrange equations are precisely the Einstein equations for the associated metric. For all theories in the considered class, the field equations are expressed entirely in terms of the exterior derivative, without explicit reference to the Levi-Civita connection.
\end{abstract}




\end{frontmatter}

\section{Introduction}

A ${\rm Spin}(7)$-structure on an 8-dimensional manifold is defined to be a 4-form of a special algebraic type. Such a 4-form is known as a Cayley form, and its ${\rm GL}(8,\R)$ stabiliser is ${\rm Spin}(7)$. An 8-manifold admits a ${\rm Spin}(7)$-structure if it is spin. However, since most of the considerations in this paper are local, we do not need to concern ourselves with  assumptions about $M$. 

As is well-known since \cite{Fernandez}, a ${\rm Spin}(7)$-structure is integrable if and only if the associated Cayley 4-form $\Phi$ is closed $d\Phi=0$. This in turn implies that the metric determined by $\Phi$ is Ricci-flat. It is clear that $d\Phi=0$ gives the most geometrically motivated set of first-order PDE's on the Cayley form. In this paper we address the question of what the most natural second-order PDE's are. We describe a certain construction, inspired by the Plebanski formalism \cite{Plebanski:1977zz}, see also \cite{Bhoja:2024xbe} for a recent description most closely aligned with the motivations of the present paper. The result of the construction is a two-parameter family of action functionals for $\Phi$, whose Euler-Lagrange equations are a set of second-order PDE's that possess some desirable properties. In particular, they are constructed solely from the operator of the exterior differentiation on forms, so one never needs to know the  covariant derivative of the metric determined by $\Phi$ to write them down. 

Our construction starts with the action
\be\label{action-intr-kappa}
S_\kappa[\Phi,C] = \int_M \Phi\wedge (dC - 6 C\wedge_\Phi C) + \frac{\kappa}{6}  (C)^2 v_\Phi + \frac{\lambda}{6} v_\Phi + \text{constr.}
\ee
Here $\Phi\in \Lambda^4(M)$ is a Cayley form, and we have included in the action a set of constraint terms whose purpose is to guarantee that $\Phi$ is of the correct algebraic type. These depend solely on $\Phi$ as well as some necessary Lagrange multipliers variation with respect to which imposes the constraints. An easy comparison between the dimension of the space of 4-forms ${\rm dim}(\Lambda^4) = 70$ and the dimension of the orbit ${\rm dim}( {\rm GL}(8,\R)/{\rm Spin}(7)) = 43$ shows that there are 27 independent constraints to be satisfied. We will never need to specify these constraints explicitly, as only the variation of these terms with respect to $\Phi$ matters for the Euler-Lagrange equations, and this can be determined by a different argument, see below. The object $C\in \Lambda^3(M)$ is what we refer to as the auxiliary 3-form. The term $\lambda v_\Phi$ in (\ref{action-intr-kappa}) is a 'cosmological constant' term, with $\lambda\in \R$ being a parameter and $v_\Phi$ being the volume form for $\Phi$, which can be taken to be $v_\Phi = (1/14) \Phi\wedge \Phi$. The quantity $\kappa\in \R$ is another parameter. Finally, $C\wedge_\Phi C$ is a 4-form constructed from two copies of $C$, as well as the (inverse) metric $g^{ab}$ determined by $\Phi$. In index notation that we will be using in this article, it is given by 
\be
(C\wedge_\Phi C)_{abcd} := C_{[ab|p|} C_{cd]q} g^{pq}.
\ee

The Euler-Lagrange equations for $C$ are algebraic, and determine $C$ in terms of the intrinsic torsion of the ${\rm Spin}(7)$-structure, or equivalently in terms of the exterior derivative of $\Phi$, see below. After this solution is substituted back into the action, one gets functional that depends solely on $\Phi$, and is given by a linear combination of the squares of the irreducible components of intrinsic torsion
\be\label{action-torsion-intr}
S_\kappa[\Phi] = \frac{1}{6} \int_M \left( \frac{(6+\kappa)|T_{48}|^2 +36(\kappa-1) |T_8|^2 }{6-(5+\kappa)\kappa} +\lambda \right) v_\Phi.
\ee
Here $T_{8,48}$ are the two irreducible components of the intrinsic torsion of $\Phi$. This functional can also be written, see (\ref{action-kappa-phi}), in terms of quantities $\Phi \wedge T\wedge_\Phi T$ and $(T)^2 v_\Phi$, where $T$ is the torsion 3-form.

Our analysis of the linearised theory in Section \ref{sec:lin} will show that there is a two-parameter family of diffeomorphism-invariant Lagrangians that are second order in derivatives and quadratic in perturbations of the Cayley form. In Section \ref{sec:lin-action} we will verify that the linearisation of (\ref{action-torsion-intr}) reproduces the two-parameter family of linearised Lagrangians of Section \ref{sec:lin}, thus showing that (\ref{action-torsion-intr}) gives the non-linear completion of the most general diffeomorphism-invariant linear Lagrangian. 

Our main aim in this paper is to characterise the Euler-Lagrange equations following from (\ref{action-intr-kappa}) and thus (\ref{action-torsion-intr}), as well as determine values of the parameter $\kappa$ that result in special properties. We identify several such special values 
\[
\kappa=0,\quad -2, \quad \frac{6}{5},
\] 
as well as values $\kappa=1, -6$ when the relation between $C$ and $T$, see (\ref{C-T-intr}), is no longer invertible, and the functional (\ref{action-torsion-intr}) is undefined. Our results are best described by the following series of propositions. 

The fact that the dimension of the space where the intrinsic torsion of a ${\rm Spin}(7)$-structure lies is equal to the dimension of the space of 3-form is known. However, the paper \cite{Spiro-Spin7}, which was an important precursor to our construction, uses a different parametrisation. The analog of Lemma 2.10 of \cite{Spiro-Spin7} in our parametrisation is the following statement:
\begin{proposition} The intrinsic torsion of a ${\rm Spin}(7)$-structure, measured by $\nabla_i \Phi_{abcd}$, where $\nabla$ is the covariant derivative with respect to the Levi-Civita connection for the metric defined by $\Phi$, lies in $\Lambda^1 \otimes \Lambda^4_7$. For the notation explaining $\Lambda^4_7$ and the decomposition of the space of forms into irreducible components see below. The intrinsic torsion can be parametrised by an object $T\in \Lambda^3$ so that
\be\label{torsion-intr}
\nabla_i \Phi_{abcd} = 4T_{i[a}{}^p \Phi_{|p|bcd]}.
\ee
Here the index $p$ of $T_{aip}$ is raised with the metric determined by $\Phi$. 
\end{proposition}
We remark that the only metric that is used in this paper is the one defined by $\Phi$. 

It turns out that the torsion 3-form is completely determined by the exterior derivative $d\Phi$. This is the content of the following proposition:
\begin{proposition} The Hodge dual of the projection of (\ref{torsion-intr}) to the space of 5-forms can be written as
\be\label{Phi-T-intr}
\star d\Phi = 2 J_3(T),
\ee
where $J_3$ is a certain operator $J_3:\Lambda^3\to \Lambda^3$ defined by $\Phi$, see (\ref{J3}). The operator $J_3$ is invertible, and so $T$ is completely determined by $d\Phi$. 
\end{proposition}

We now explain what makes the $\kappa=0$ version of the functional (\ref{action-intr-kappa}) special. 
\begin{proposition}
The Euler-Lagrange equation arising from $S_{\kappa=0}[\Phi,C]$ by extremising it with respect to $C$ is $C=T$.
\end{proposition}
One can rephrase this by saying that $S_{\kappa=0}[\Phi,C]$ is precisely the first-order action dependent on both $\Phi,C$ that leads to $C=T$ as the $C$ field equation. Importantly, there is no ambiguity in the construction of the action once we demand that $C=T$ is to follow. In contrast, the critical value of $C$ for the more general $\kappa\not=0$ action (\ref{action-intr-kappa}) is 
\be\label{C-T-intr}
C= \frac{6 T +\kappa J_3(T)}{6 - (5+\kappa)\kappa}.
\ee
It is clear that $\kappa=1,-6$ are special values in that the relation between $C$ and $T$ is no longer invertible. 
The property of the action $S_{\kappa=0}[\Phi,C]$ that the value of $C$ as determined by its corresponding Euler-Lagrange equation is the intrinsic torsion $C=T$ makes this action a precise analogue of the 4D Plebanski action, see \cite{Bhoja:2024xbe}. In this sense, it is a preferred member in the family of action functionals (\ref{action-intr-kappa}). 

The next proposition describes the Euler-Lagrange equations resulting by varying $S_{\kappa=0}[\Phi,C]$ with respect to $\Phi$:
\begin{proposition}
The Euler-Lagrange equations resulting from extremisation of $S_{\kappa=0}[\Phi,C]$ with respect to $\Phi$ can be written as:
\be\label{Phi-eqs-intr}
\partial_{[a}T_{bcd]} - \frac{3}{2}  T_{[ab}{}^{p} T_{cd]p}  - \frac{1}{8} (TT)_{[a|e|} \Phi^e{}_{bcd]} + \frac{\lambda}{84} \Phi_{abcd} = \Psi_{[ab}{}^{pq} \Phi_{|pq| cd]}, 
\ee
where $\Psi^{abcd}$ is an arbitrary matrix in ${\rm Sym}_0^2(\Lambda^2_7)$, which is the space of symmetric tracefree matrices, with the trace defined as ${\rm Tr}(\Psi)=\Psi_{ab}{}^{ab}$, and 
\be
(TT)_{ab}:= \Phi^{ijkl} T_{ija} T_{klb} - \frac{1}{7} g_{ab} \Phi^{ijkl} T_{ij}{}^p T_{klp}
\ee
is a symmetric matrix quadratic in the torsion. We note that the factor of $1/7$ here is not a typo. We can also write the Euler-Lagrange equations in form notation as
\be\label{feqs-form-intr}
d\Phi - 6 T\wedge_\Phi T - \frac{1}{16} K(TT) + \frac{\lambda}{42} \Phi = \Psi\Phi. 
\ee
Here $K$ is the map from the space of symmetric tensors to $\Lambda^4$ described in (\ref{h-delta-phi}). It is shown later in the text, see Section 2.9, that a general 4-form in $\Lambda^4_{27}$ can be parametrised as $\Psi\Phi$, and so the right-hand side of both (\ref{Phi-eqs-intr}) and (\ref{feqs-form-intr}) is a general element of  $\Lambda^4_{27}$. 
\end{proposition}

\begin{proposition} An alternative way of writing the field equations is to project both sides on the $\Lambda^4_{1+7+35}$ component in $\Lambda^4$. This gives the following set of equations
\be\label{feqs-intr}
\frac{1}{4} \Phi_{b}{}^{pqr} \nabla_{a}T_{pqr} - \frac{3}{4}  \Phi_{b}{}^{pqr}\nabla_r T_{apq}
- \frac{3}{2}\Phi_{b}{}^{pqr} T_{ap}{}^s T_{qrs}  \\ \nonumber
- \frac{3}{4}\Phi^{pqrs} T_{apq} T_{brs}
+ \frac{1}{2} g_{ab}\left(\lambda +\frac{3}{8} \Phi^{pqrs} T_{pq}{}^p T_{rsp}\right)=0.
\ee
The equations here are written in terms of $\nabla$, but they have the same form with $\nabla$ replaced by the partial derivative operator $\partial$. We note that the left-hand side is not automatically $ab$ symmetric, and the anti-symmetric part of these equations are non-trivial. The anti-symmetric part can be shown to lie in $\Lambda^2_7$, and so the total number of independent second-order differential equations is $36+7=43$, the dimension of the space of Cayley forms. 
\end{proposition}
We can also characterise what the $\kappa=0$ field equations imply for the Riemann curvature of the metric defined by $\Phi$. This is done in the main text. We will see that the metrics defined by $\Phi$ that are the critical points of the action functional $S_{\kappa=0}[\Phi,C]$ are not in general Einstein. 

There is another special Lagrangian in the family (\ref{action-intr-kappa}), namely one corresponding to $\kappa=-2$. 
\begin{proposition}
The functional $S_{\kappa=-2}[\Phi]$ is a multiple of the integrated scalar curvature, plus a volume term. More precisely
\[
S_{\kappa=-2}[\Phi] = \frac{1}{6} \int \Big( \frac{4}{3} T J_3(T)  -T^2    + \lambda \Big) v_g d^8x 
= - \frac{1}{6} \int \Big( R -  \lambda \Big)v_g d^8x.
\]
Here $v_g=v_\Phi$ is the volume form for the metric defined by $\Phi$. The resulting Euler-Lagrange equations are Einstein equations for the metric defined by $\Phi$. 
\end{proposition}
Thus, Einstein metrics in eight dimensions can be described as critical points of the $\kappa=-2$ action. 

It is somewhat surprising that there appear two different eight dimensional generalisations of the Plebanski formalism for 4D General Relativity, see \cite{Bhoja:2024xbe} for the presentation most closely aligned to one in this paper. One such generalisation is given by the $\kappa=0$ action, which shares with the 4D Plebanski action the property that the Euler-Lagrange equations for the auxiliary field sets it to be equal to the intrinsic torsion. Another generalisation is the $\kappa=-2$ functional, which shares with the 4D Plebanski action the property that the critical points are Einstein structures.  

For any value of $\kappa$, the gradient of the action functional $S_\kappa[\Phi]$ with respect to $\Phi$ defines a certain ${\rm Spin}(7)$-flow whose properties are yet to be understood. In particular, for $\kappa=0$ this flow is given by the Hodge dual of the 4-form on the left-hand-side of (\ref{feqs-form-intr}), projected to the space $\Lambda^4_{1+7+35}$. It is through this gradient flow that our work connects to a wider context of flows of geometric structures \cite{Moreno-Earp}. Much work has been done in the context of such flows for the case of $G_2$-structures, starting with \cite{Bryant} and culminating in the recent work \cite{Spiro-G2}, which was an important predecessor to this paper. A particular geometric flow that has received attention is the so-called harmonic flow \cite{Loubeau-Earp1}, which is defined as the gradient flow of the squared norm of the torsion tensor. This flow has been studied in various settings in \cite{Fowdar-Earp1}, \cite{Fowdar-Earp2}, \cite{Dwivedi-Earp}, \cite{Loubeau} and very recently in the case of ${\rm Spin}(7)$-structures in \cite{Dwivedi}. The functional generating the flow in \cite{Dwivedi} is $\int |T|^2$. As is clear from (\ref{action-kappa-phi}), this corresponds to $\kappa=6/5$. 
In contrast, the functional $S_{\kappa=0}[\Phi]$ that arises if we demand $C=T$ is given by 
\[
S_{\kappa=0}[\Phi]  \sim \int_M  T J_3(T) v_\Phi =\int_M (|T_{48}|^2 - 6 |T_8|^2) v_\Phi.
\]
Unlike $\int |T|^2$, this functional does not have a definite sign. The other natural functional, whose critical points are Einstein structures, namely
\[
S_{\kappa=-2}[\Phi] \sim \int_M (\frac{4}{3} T J_3(T) - T^2) v_\Phi = \int_M (\frac{1}{3} |T_{48}|^2 - 9 |T_8|^2) v_\Phi,
\]
also does not have a definite sign. 

More work is needed to get better intuition about the properties of the critical points of both the $\kappa=0$ and $\kappa\not=0$ actions. We hope that this work will follow. It is worth remarking already at this point, however, that a construction similar to that described in this paper is possible also for other $G$-structures, in various dimensions. It would be particularly interesting to perform a similar analysis and construct actions for 3-forms in 7-dimensions, building on the work \cite{Spiro-G2}. 

Many of the tensor computations in the paper are performed using symbolic manipulation with xAct Mathematica package \cite{xact}. For the convenience of the interested reader, we have placed a Mathematica notebook containing all the definitions and some example calculations in a GitHub repository \cite{Git}. The stated representation theoretic facts are obtained using the Mathematica package LieART \cite{Feger:2019tvk}. 

\section{Decomposition of the spaces of forms}

Our index (anti-) symmetrisation conventions are as follows. Square brackets denote anti-symmetrisation, and are defined by 
\be
[a_1 \ldots a_k] = \frac{1}{k!} \sum_p (-1)^p p(a_1 \ldots a_k),
\ee
where the sum is taken over all permutations $p$ and $(-1)^p$ is plus or minus identity depending on whether the permutation is even or odd. Symmetrisation is denoted by round brackets, and is defined similarly, apart from the absence of the minus signs. A vertical line around an index (or a set of indices) denotes the fact that this set of indices is not involved in (anti-) symmetrisation. 

\subsection{Basic algebra}

Similar to \cite{Bryant} and \cite{Spiro-Spin7}, we use the index notation, which is very useful for encoding various relations satisfied by the Cayley form. 
The basic algebraic relation satisfied by the 4-form $\Phi_{abcd}$ is
\be\label{phi-single-contr}
\Phi_{ijkp} \Phi_{abcp} = \\ \nonumber g_{ia} g_{jb} g_{kc} + g_{ib} g_{jc} g_{ka} +g_{ic} g_{ja} g_{kb} 
- g_{ia} g_{jc} g_{kb} -g_{ic} g_{jb} g_{ka} -g_{ib} g_{ja} g_{kc} \\ \nonumber
- g_{ia} \Phi_{jkbc} - g_{ja} \Phi_{kibc}-g_{ka} \Phi_{ijbc} \\ \nonumber
- g_{ib} \Phi_{jkca} - g_{jb} \Phi_{kica}-g_{kb} \Phi_{ijca}\\ \nonumber
- g_{ic} \Phi_{jkab} - g_{jc} \Phi_{kiab}-g_{kc} \Phi_{ijab}.
\ee
This identity, together with identities related to self-duality of $\Phi$, are sufficient for most of the calculations one needs to do with $\Phi$. A Mathematical notebook that can be used for algebraic computations with $\Phi$ based on this identity is available via \cite{Git}. 

One more contraction of the above identity gives
\be\label{phi-double-contr}
\Phi_{ijpq} \Phi_{abpq} = 6 g_{ia} g_{jb} - 6 g_{ib} g_{ja} - 4\Phi_{ijab}.
\ee
Yet one more contraction gives
\be\label{phi-triple-contr}
\Phi_{ipqr} \Phi_{apqr} = 42 g_{ia}.
\ee

The 4-form $\Phi$ is self-dual
\be\label{self-duality}
\frac{1}{4!}\epsilon_{ijkl}{}^{abcd}\Phi_{abcd} = \Phi_{ijkl}.
\ee

Useful consequences of self-duality are
\be\label{epsilon-phi-3}
\epsilon^{aijkl pqr} \Phi_{bpqr} =  30 \delta_b^{[a} \Phi^{ijkl]},
\ee
and
\be\label{epsilon-phi-2}
 \epsilon^{ijklmn pq} \Phi_{abpq} =  60 \delta_a^{[i} \delta_b^j \Phi^{klmn]},
\ee
and
\be\label{epsilon-phi-1}
 \epsilon^{ijklmnp r} \Phi_{abc r} =  210 \delta_a^{[i} \delta_b^j \delta_c^k \Phi^{lmnp]}.
\ee

\subsection{Identity}

The following non-trivial identity 
\be\label{identity}
- 2\Phi_{[ijk}{}^{[a} \Phi_{l]}{}^{bcd]} -3 \Phi_{[ij}{}^{[ab} \Phi_{kl]}{}^{cd]} +42 \Phi_{[ij}{}^{[ab} \delta_k^c \delta_{l]}^{d]} + \Phi_{ijkl} \Phi^{abcd}=0
\ee
can be checked by multiplying with $\delta_a^i$ and using the identity (\ref{phi-single-contr}) to check that the result is zero. Another useful check is to contract the left-hand side with $\Phi^{ijkl}$, again producing zero. This identity does not seem to have appeared in the literature before. One can derive this identity as follows. An argument detailed in Section 2.5 below suggests that there must exist an identity of this sort. One can then take an arbitrary combination of the first three terms and equate this to the last term. Requiring the contraction with $\delta_a^i$ and with $\Phi^{ijkl}$ to give correct identities leads to the unique choice of the coefficients as above. This computation is spelled out in \cite{Git}. Another possibility is to look for an arbitrary linear combination of the first three terms that is self-dual with respect to both $abcd$ and $ijkl$. Again, this uniquely determines the coefficients as above. 

\subsection{Decomposition of $\Lambda^2$}

The following material is standard, see e.g. \cite{Spiro-Spin7}. Our notation for the operators introduced is different from that in \cite{Spiro-Spin7}, but, we hope, is systematic and convenient. 

We introduce the following operator on 2-forms
\be
J_2: \Lambda^2\to \Lambda^2, \qquad  J_2(\beta)_{ij} = \frac{1}{2} \Phi_{ij}{}^{ab} \beta_{ab}.
\ee
Using (\ref{phi-double-contr}) we see that
\be
(J_2)^2 = 3\mathbb{I} - 2J_2.
\ee
This means that the eigenvalues of $J_2$ are $-3, 1$. The eigenspace of eigenvalue $-3$ is $\Lambda^2_7$, and eigenspace of eigenvalue $1$ is $\Lambda^2_{21}$. The two projectors are
\be
\pi_7 = \frac{1}{4}(\mathbb{I} - J_2), \qquad \pi_{21} = \frac{1}{4}(3\mathbb{I} + J_2).
\ee
For later purposes we note that
\be
J_2^{-1} = \frac{1}{3}( 2\mathbb{I} + J_2).
\ee

\subsection{Decomposition of $\Lambda^3$}

We introduce the following operator on 3-forms
\be\label{J3}
J_3:\Lambda^3\to \Lambda^3, \\ \nonumber
J_3(\gamma)_{ijk} = \frac{1}{2}( \Phi_{ij}{}^{pq} \gamma_{kpq} +  \Phi_{jk}{}^{pq} \gamma_{ipq}+ \Phi_{ki}{}^{pq} \gamma_{jpq}) = \frac{3}{2} \Phi_{[ij}{}^{pq} \gamma_{k]pq} 
\ee
A calculation using (\ref{phi-single-contr}) gives
\be
(J_3)^2 = 6 \mathbb{I} - 5 J_3.
\ee
This means that the eigenvalues of $J_3$ are $-6,1$. The eigenspace of eigenvalue $-6$ is $\Lambda^3_8$, and eigenspace of eigenvalue $1$ is $\Lambda^3_{48}$. The elements of the space $\Lambda^3_8$ are of the form
\be\label{lambda3-8}
\Lambda^3_8 = \{ X^p \Phi_{pijk}, X\in TM\},
\ee
and 
\be\label{lambda3-48}
\Lambda^3_{48} = \{ \gamma\in \Lambda^3 : \gamma\wedge \Phi=0\}.
\ee
We note that
\be\label{projector-48}
\pi_{48} = \frac{6}{7} \left( \mathbb{I} + \frac{1}{6} J_3\right), \qquad \pi_8 = \frac{1}{7} \left( \mathbb{I}-J_3\right).
\ee
We also note that
\be\label{J3-inv}
J_3^{-1} = \frac{1}{6}(J_3+5\mathbb{I}).
\ee

\subsection{Decomposition of $\Lambda^4$}

Unlike \cite{Spiro-Spin7}, which uses a more indirect approach, we decompose $\Lambda^4$ in a way completely analogous to what was done in the case $\Lambda^2, \Lambda^3$. The only arising difficulty is that the operator that we need in this case satisfies a more complicated (fourth-order) relation. 

We introduce the following operator on 4-forms
\be 
J_4:\Lambda^4\to \Lambda^4, \qquad J_4(\sigma)_{ijkl} = 3 \Phi_{[ij}{}^{pq} \sigma_{kl]pq} = \\ \nonumber
\frac{1}{2}( \Phi_{ij}{}^{pq} \sigma_{pqkl} + \Phi_{ki}{}^{pq} \sigma_{pqjl} + \Phi_{il}{}^{pq} \sigma_{pqjk} 
+ \Phi_{kl}{}^{pq} \sigma_{pqij} + \Phi_{jl}{}^{pq} \sigma_{pqki} + \Phi_{jk}{}^{pq} \sigma_{pqil} ).
\ee
We remark that this is the map denoted by $\Lambda_\Phi$ in \cite{Spiro-Spin7}.
We have the following relation, also to be found in \cite{Spiro-Spin7}
\be\label{J4-squared}
(J_4)^2(\sigma)_{ijkl} = \frac{1}{2}( \Phi_{ij}{}^{ab} \Phi_{kl}{}^{cd} + \Phi_{ki}{}^{ab} \Phi_{jl}{}^{cd}+\Phi_{il}{}^{ab} \Phi_{jk}{}^{cd})\sigma_{abcd} 
\\ \nonumber
+ 6 \sigma_{ijkl} - 8 J_4(\sigma)_{ijkl} = 
\frac{3}{2} \Phi_{[ij}{}^{ab} \Phi_{kl]}{}^{cd} \sigma_{abcd} - 24 \Phi_{[ij}{}^{ab} \sigma_{kl]ab}+ 6\sigma_{ijkl}.
\ee
We also have the following result for the cube of this operator
\be
(J_4)^3(\sigma)_{ijkl} = -6 \Phi_{[ijk}^a \Phi_{l]}^{bcd} \sigma_{abcd} - 15 \Phi_{[ij}{}^{ab} \Phi_{kl]}{}^{cd} \sigma_{abcd} \\ \nonumber
+ 258 \Phi_{[ij}{}^{ab} \sigma_{kl]ab} 
-24 \sigma_{ijkl}.
\ee
Using the identity (\ref{identity}) we can rewrite this as
\be\label{J4-cube}
(J_4)^3(\sigma)_{ijkl} =  -6 \Phi_{[ij}{}^{ab} \Phi_{kl]}{}^{cd} \sigma_{abcd} + 132 \Phi_{[ij}{}^{ab} \sigma_{kl]ab} \\ \nonumber
- 3\Phi_{ijkl} \Phi^{abcd} \sigma_{abcd} 
-24 \sigma_{ijkl}.
\ee
Finally, for the fourth power of this operator we have
\be
(J_4)^4(\sigma)_{ijkl} = 87 \Phi_{[ijk}^a \Phi_{l]}^{bcd} \sigma_{abcd} +\frac{345}{2} \Phi_{[ij}{}^{ab} \Phi_{kl]}{}^{cd} \sigma_{abcd} \\ \nonumber 
-2643 \Phi_{[ij}{}^{ab} \sigma_{kl]ab} 
+ \frac{9}{2} \Phi_{ijkl} \Phi^{abcd} \sigma_{abcd}+168 \sigma_{ijkl}.
\ee
Using (\ref{identity}) we can rewrite this as
\be
(J_4)^4(\sigma)_{ijkl} =  42 \Phi_{[ij}{}^{ab} \Phi_{kl]}{}^{cd} \sigma_{abcd} -816 \Phi_{[ij}{}^{ab} \sigma_{kl]ab} \\ \nonumber 
+48 \Phi_{ijkl} \Phi^{abcd} \sigma_{abcd} 
+168 \sigma_{ijkl}.
\ee
This shows that 
\be\label{J4-identity}
(J_4)^4 + 16 (J_4)^3+ 36 (J_4)^2 - 144 J_4=0,
\ee
or in other words
\be
(J_4+12\mathbb{I})(J_4+6\mathbb{I})(J_4-2\mathbb{I})J_4=0.
\ee
This shows that the operator $J_4$ has eigenvalues $-12, -6, 2, 0$, see also \cite{Spiro-Spin7}. It is clear that the identity (\ref{identity}) is key to make this calculation work. In fact, this is how the identity (\ref{identity}) was derived in the first place. One knows that (\ref{J4-identity}) must be true, which shows that an identity of the type (\ref{identity}) must be true. One then writes a general relation of the sort (\ref{identity}), and fixes the coefficients in such a way that its contraction gives a true statement. This uniquely fixes (\ref{identity}). 

The eigenspaces are the irreducible components of the space of 4-forms
\be\label{J4-eigenspaces}
\Lambda^4_1 =\{ \sigma\in \Lambda^4: J_4(\sigma)=-12 \sigma\}, \quad \Lambda^4_{27} =\{ \sigma\in \Lambda^4: J_4(\sigma)=2 \sigma\}, \\ \nonumber
\Lambda^4_7 =\{ \sigma\in \Lambda^4: J_4(\sigma)=-6 \sigma\}, \quad \Lambda^4_{35} =\{ \sigma\in \Lambda^4: J_4(\sigma)=0\}.
\ee
This will follow after we characterise each of the irreducible components below.  

\subsection{Projector on $\Lambda^4_{27}$}

In what follows it will be useful to have the projector on $\Lambda^4_{27}$ explicitly. It is clear that it is a multiple of $J_4(J_4+12\mathbb{I})(J_4+6\mathbb{I})$. Taking into account the eigenvalues of $J_4$ on different subspaces, it is not difficult to check that the required multiple is $1/224$. Thus, we have
\be
\pi_{27} = \frac{1}{224} J_4(J_4+12\mathbb{I})(J_4+6\mathbb{I}) = \frac{1}{224}( (J_4)^3+ 18 (J_4)^2 + 72 J_4).
\ee
Using (\ref{J4-cube}) and (\ref{J4-squared}) we get
\be\label{projector-27}
\pi_{27}(\sigma)_{ijkl} = 
\frac{3}{32}\Big(  \Phi_{[ij}{}^{ab} \Phi_{kl]}{}^{cd} \sigma_{abcd} -4 \Phi_{[ij}{}^{ab} \sigma_{kl]ab} \\ \nonumber
-\frac{1}{7} \Phi_{ijkl} \Phi^{abcd} \sigma_{abcd} +4 \sigma_{ijkl} \Big) .
\ee
We have explicitly checked that this projector kills 4-forms of the form
\be
H_{[i}{}^p \Phi_{jkl]p}, \qquad H\in \Lambda^1\otimes \Lambda^1,
\ee
which lie in $\Lambda^4_{1+35+7}$. This characterisation of $\Lambda^4_{1+35+7}$ is subject of the next two subsections. 

For later purposes, we note that $\mathbb{I}-\pi_{27}$ projects out the $\Lambda^4_{27}$ component of any 4-form, and is given by
\be\label{proj-away-27}
\mathbb{I}-\pi_{27}= \frac{1}{32}\Big( 20 \sigma_{ijkl} -3 \Phi_{[ij}{}^{ab} \Phi_{kl]}{}^{cd} \sigma_{abcd} \\ \nonumber 
+12 \Phi_{[ij}{}^{ab} \sigma_{kl]ab} +\frac{3}{7} \Phi_{ijkl} \Phi^{abcd} \sigma_{abcd} \Big) .
\ee

We also note that $\pi_{27}$ can be understood in a simple way. Indeed, taking a 4-form $\sigma_{ijkl}\in \Lambda^4$, we can interpret this as an object in ${\rm Sym}^2(\Lambda^2)$, and apply the projector $\pi_7$ on the indices $ij$ and on the indices $kl$. After this the result can be projected back to $\Lambda^4$ by antisymmetrising the indices. The result of this operation is
\be
\left((\pi_7 \sigma \pi_7)\Big|_{\Lambda^4}\right)_{ijkl} = \frac{1}{64} \left(  \Phi_{[ij}{}^{ab} \Phi_{kl]}{}^{cd} \sigma_{abcd} -4 \Phi_{[ij}{}^{ab} \sigma_{kl]ab} + 4 \sigma_{ijkl} \right) .
\ee
This contains almost all the terms in $\pi_{27}(\sigma)$. The only term present in (\ref{projector-27}) and absent in $(\pi_7 \sigma \pi_7)\Big|_{\Lambda^4}$ is the third term in (\ref{projector-27}), whose purpose is to make the result tracefree. So, we can write
\be
\pi_{27}(\sigma) = 6 (\pi_7 \sigma \pi_7)\Big|_{\Lambda^4} - \text{trace},
\ee
where the last term just removes the trace of the first. This gives a simple and useful interpretation of the projector $\pi_{27}$.

\subsection{An operator from $\Lambda^2$ to $\Lambda^4_7$ and its inverse}

Let us introduce the following operator
\be\label{operator-K}
\Lambda^2\ni \beta_{ij} \to K(\beta)_{ijkl} =  4\beta_{[i|p|} \Phi^p{}_{jkl]} \in \Lambda^4.
\ee
We note that the introduced operator $K$ is the diamond map from \cite{Spiro-Spin7}. It can now be checked that
\be
K\circ\pi_{21}=0, \qquad K\circ\pi_7 = K.
\ee
Both of these can be understood by noting that the image $K(\beta)\in \Lambda^4$ is precisely the orbit of the basic 4-form $\Phi$ under the action of the Lie algebra $\mathfrak{spin}(8)$. The statement that $K\circ \pi_{21}=0$ is just the statement that $\Phi$ is ${\rm Spin}(7)$ invariant. This means that ${\rm Ker}(K)=\Lambda_{21}^2$, and the image is $\Lambda^4_7$. 

To find the inverse of $K$ on $\Lambda^4_7$ let us consider 
\be\label{operator-K-prime}
\Lambda^4\ni \sigma_{ijkl} \to K'(\sigma)_{ij} = \frac{1}{2} \Phi_{i}{}^{pqr} \sigma_{jpqr} -\frac{1}{2} \Phi_{j}{}^{pqr} \sigma_{ipqr} \in \Lambda^2.
\ee
We then have 
\be
\pi_{21}\circ K' =0, \qquad K'\circ K = 96 \pi_7.
\ee
This means that $K'$ is (a multiple of) the inverse of $K$ on $\Lambda^4_7$. 

The operator $K$ to $\Lambda^4$ can be generalised and applied to a general tensor from $\Lambda^1\otimes \Lambda^1$. In particular, it can be applied to a symmetric tensor $h_{ij}\in {\rm Sym}^2(\Lambda^1)$
\be\label{h-delta-phi}
{\rm Sym}^2(\Lambda^1) \ni h_{ij} \to K(h)_{ijkl} = 4h_{[i|p|} \Phi^p{}_{jkl]}\in \Lambda^4.
\ee
The image of ${\rm Sym}^2(\Lambda^1)$ under the action of $K$ can be seen to be $\Lambda^4_1\oplus \Lambda^4_{35}$. Restricted to its image, the operator $K: {\rm Sym}^2(\Lambda^1)\to \Lambda^4$ is invertible, with the inverse being a multiple of
\be
K' : \Lambda^4 \to {\rm Sym}^2(\Lambda^1), \qquad  K'(\sigma)_{ij} =  \Phi_{(i}{}^{pqr} \sigma_{j)pqr} \in {\rm Sym}^2(\Lambda^1).
\ee

\subsection{Characterisation of $\Lambda^4_{1+35+7}$}

We can apply the map $K$ to a general element $H_{ij}\in \Lambda^1\otimes\Lambda^1$
\be
K(H)_{ijkl} := 4 H_{[i{}}^p \Phi_{jkl]p}.
\ee
We already know that $\pi_{27}( K(H)) =0$, and so the image of this map lies in $\Lambda^4_{1+35+7}$. We also know that the map $K$ applied to the symmetric part of $H$ lies in $\Lambda^4_{1+35}$, and to the anti-symmetric part in $\Lambda^4_7$. A computation gives  the following result
\be
J_4(K(H))_{ijkl} = -3 (H_{[i}{}^p - H^p{}_{[i}) \Phi_{ijk]p} -6 \Phi_{ijkl} H_p{}^p.
\ee
This shows that when $H$ is symmetric tracefree $H_{ij}=H_{(ij)}, H_p{}^p=0$ we have $J_4(K(H))=0$. This shows that $\Lambda^4_{35}$ is the eigenspace of $J_4$ of eigenvalue $0$. When $H_{ij}$ is anti-symmetric, we have $J_4(K(H)) = - 6 K(H)$, and so $\Lambda^4_7$ is eigenspace of eigenvalue $-6$. When $H_{ij}=g_{ij}$, we have $J_4(K(H))= - 12 K(H)$, and thus $\Lambda^4_1$ is eigenspace of eigenvalue $-12$. This gives the characterisation described above in (\ref{J4-eigenspaces}). 

\subsection{Characterisation of $\Lambda^4_{27}$}

It will be useful to have an explicit parametrisation of a general element of $\Lambda^4_{27}$, similar to have we already have a parametrisation of a general element of the other irreducible subspaces $\Lambda^4_{1+35+7}$. We start with a definition
\begin{definition}
We say that a matrix $\Psi$ is an element of the space of symmetric tracefree matrices
\be
\Psi^{ab cd}\in {\rm Sym}_0^2(\Lambda^2_7),
\ee
if the following requirements are satisfied:
\be
\Psi^{ab cd}= \Psi^{[ab] [cd]} = \Psi^{[cd] [ab]}, \qquad \Psi_{ab}{}^{ab}=0,  \qquad \pi_7 \Psi = \Psi = \Psi \pi_7.
\ee
\end{definition}

We can then construct a 4-form that we denote as $\Psi\Phi$ as
\be\label{Psi-Phi}
(\Psi \Phi)_{abcd} := \Psi_{[ab}{}^{pq} \Phi_{cd]pq}.
\ee
Let us show that the projection of this to $\Lambda^4_{1+35+7}$ vanishes. To compute this, we apply the map $K': \Lambda^4 \to \Lambda^1 \otimes \Lambda^1$. A calculation gives
\be\label{Psi-Phi-proj}
(\Psi \Phi)_{ipqr} \Phi^{apqr} = \delta_i^a\left( \Psi_{qr}{}^{qr} - \frac{1}{2} \Psi_{pqrs} \Phi^{pqrs}\right) \\ \nonumber 
+ 4 \Psi_{ip}{}^{ap} + \Phi_{ipqr} \Psi^{apqr} -  \Psi_{ipqr} \Phi^{apqr}.
\ee
The first term here is zero because $\Psi$ is tracefree  $\Psi_{qr}{}^{qp} =0$, and also $\pi_7 \Psi = \Psi$ means $\pi_{21} \Psi=0$, which implies
\be
3\Psi_{ijkl}+ \frac{1}{2}\Phi_{ij}{}^{pq} \Psi_{pqkl} =0.
\ee
So, if $\Psi_{qr}{}^{qp} =0$ then also $\Psi_{pqrs} \Phi^{pqrs}=0$. On the other hand, if we contract $jl$ in this expression we get
\be\label{Phi-Psi-contr}
\Psi_{k}{}^{pqr} \Phi_{ipqr} = - 6 \Psi_{ipk}{}^p.
\ee
The right-hand side is $ik$ symmetric, and thus the left-hand side must also be $ik$ symmetric. This shows that the last two terms in (\ref{Psi-Phi-proj}) cancel. It remains to characterise $\Psi_{ip}{}^{ap}$. To do this, we compute $0=\pi_{21} \Psi\pi_{21}$
\be
0= 3\Psi_{ijkl} + \frac{3}{2} \Phi_{ij}{}^{pq} \Psi_{pq kl} + \frac{3}{2} \Psi_{ij}{}^{pq} \Phi_{pqkl} + \frac{1}{4} \Phi_{ij}{}^{pq} \Phi_{kl}{}^{rs} \Psi_{pqrs}.
\ee
Taking the $jl$ contraction of this we get
\be
0= 8 \Psi_{ipk}{}^p + 2\Psi_{k}{}^{pqr} \Phi_{ipqr} + 2\Psi_{i}{}^{pqr} \Phi_{kpqr} \\ \nonumber 
+ \frac{1}{2} g_{ik} \left( \Psi_{pq}{}^{pq} - \frac{1}{2} \Psi_{pqrs} \Phi^{pqrs}\right).
\ee
Using $\Psi_{qr}{}^{qp} =0, \Psi_{pqrs} \Phi^{pqrs}=0$ as well as (\ref{Phi-Psi-contr}) here we see that $\Psi_{ipk}{}^p=0$. All in all, all the terms in (\ref{Psi-Phi-proj}) are zero and the object  (\ref{Psi-Phi}) is in $\Lambda^4_{27}$. This gives the desired parametrisation of a general element of $\Lambda^4_{27}$.

\section{Linearised theory}
\label{sec:lin}

In this section we address the question as to what is the most general action invariant under diffeomorphisms that can be constructed for the fields living in the representations $\Lambda^4_{1+7+35}$ of the group ${\rm Spin}(7)$. We use physics terminology here, in which a field is a tensor that transforms in some (not necessarily irreducible) representation of the relevant Lie group. In our case the fields in $\Lambda^4_{1+35}\sim {\rm Sym}^2(\Lambda^1)$ encode perturbations $h_{ab}$ of a metric tensor, and $\Lambda^4_{7}\sim \Lambda^2_7$ is an additional field. It is well-known that there is a unique quadratic in $h_{ab}$ and second-order in derivatives diffeomorphism invariant action. This holds true in any dimension, and the argument to this effect will be given below. We will see that, in contrast, there is no longer a unique diffeomorphism-invariant action for ${\rm Sym}^2(\Lambda^1)$ and $\Lambda^2_7$ fields. There are two independent possible linearised actions that can be constructed. Non-linear completion of the theories described here is the subject of the following sections. 

\subsection{The usual metric only case}

This story is standard, and works in exactly the same way in any dimension. We review it for completeness, and for establishing the main idea of the calculation to follow in the ${\rm Spin}(7)$ case. For concreteness, we do calculations in dimension eight, but the story repeats itself with no changes in any dimension. 

In the usual gravity case one considers fields transforming with respect to the Lorentz group ${\rm SO}(8)$. The metric perturbation contains two irreducible representations ${\bf 1}, {\bf 35}_v$. The subscript $v$ stands for 'vector', to distinguish them from also possible spinor representations. This is a standard notation at least in some literature. As is also standard, we refer to the irreducible representations by their dimensions written in bold face. Let us denote the fields in representations ${\bf 1}, {\bf 35}_v$ by $h,\tilde{h}_{ab}$ respectively. We use the same letter to refer to both fields because later it will be convenient to combine them together into a single symmetric tensor $h_{ab}$. We are interested in an action that contains two derivatives. It will be useful to think in terms of Fourier transform, and denote the derivative by its Fourier transform $p^a$ at intermediate stages of the computation. The two most obvious action terms one can construct are of the type $p^2 h^2, p^2 (\tilde{h}_{ab})^2$, where the notations used are $p^2 = p^c p_c$ and $(\tilde{h}_{ab})^2 = \tilde{h}_{ab} \tilde{h}^{ab}$. To analyse the other possible terms we need to decompose the product of two derivatives into irreducibles, taking into account that they commute. We have
\be
{\bf 8}_v \otimes_S {\bf 8}_v = {\bf 1} + {\bf 35}_v.
\ee
The trivial representation here corresponds to $p^2$, and the other representation is $p_a p_b$ with the trace removed. The $p^2$ terms were already taken into account, so we only need to consider the possible couplings between $p_a p_b$ and the two other factors of either $h$ or $\tilde{h}$. There is no term with $p_a p_b$ and two factors of $h$. There is clearly a mixed term $h p^a p^b \tilde{h}_{ab}$. To determine possible terms with two factors of $\tilde{h}$ we need 
\be
{\bf 35}_v \otimes_S {\bf 35}_v = {\bf 1} + {\bf 35}_v + {\bf 294}_v + {\bf 300}.
\ee
There is only a single occurrence of ${\bf 35}_v$ here, which means that there is only a single term that does not reduce to $p^2 (\tilde{h}_{ab})^2$, and this is $(p^a \tilde{h}_{ab})^2$. 

All in all, there are just 4 possible terms in the action that one can write. We now go back to the notation that uses the derivative operators, and write the linear combination of the above four terms with arbitrary coefficients
\be
{\mathcal L} = \frac{1}{2} \tilde{h}^{bc} \partial^a \partial_a \tilde{h}_{bc} + \frac{\alpha}{2} h \partial^a \partial_a h + \beta h \partial^a \partial^b \tilde{h}_{ab} + \gamma (\partial^a \tilde{h}_{ab})^2.
\ee
A note is in order about our notations here. The notation $(t_{a_1 \ldots a_k})^2$, where $t_{a_1 \ldots a_k}$ is an arbitrary tensor, means $t_{a_1 \ldots a_k} t^{a_1 \ldots a_k}$. This explains the last term, and similar type Lagrangian terms that will be written below. 
We have chosen the coefficient in front of the first term to be $1/2$, which we can always do by changing an overall coefficient in front of the Lagrangian $\tilde{h}_{ab}$. At this stage it will be more convenient to introduce the fields
\be
h_{ab} : = \tilde{h}_{ab} + \frac{1}{8} \eta_{ab} h,
\ee
so that $\tilde{h}_{ab}$ is the tracefree part of $h_{ab}$ and $h = \eta^{ab} h_{ab}$. It is clear that the Lagrangian retains the same general form, except that the coefficients change. We will give the new coefficients the same name, hoping it will not lead to any confusion. The Lagrangian in terms of $h_{ab}$ is
\be
{\mathcal L} = \frac{1}{2} (\partial_a h_{bc})^2 + \frac{\alpha}{2}  (\partial_a h)^2 - \beta h \partial^a \partial^b h_{ab} - \gamma (\partial^a h_{ab})^2.
\ee
We now demand diffeomorphism invariance of the action, with the field transformation properties being
\be
h_{ab} = \partial_{(a} \xi_{b)}.  
\ee
Note that this encapsulates transformation properties of both $h,\tilde{h}_{ab}$. In particular $\delta h = \partial^c \xi_c$. 

We now perform the variation, and set coefficients in front of independent terms to zero, allowing integration by parts. This results in the following set of coefficients
\be
\gamma=\beta=1, \quad  \alpha = - 1.
\ee
Thus, the unique (modulo field rescaling) Lagrangian that is diffeomorphism-invariant reads
\be\label{lin-GR}
{\mathcal L}_{GR} = \frac{1}{2} (\partial_a h_{bc})^2 - \frac{1}{2}  (\partial_a h)^2 -  h \partial^a \partial^b h_{ab} -  (\partial^a h_{ab})^2,
\ee
which is the standard result. 

\subsection{Gauge-fixing}

Let us also derive the standard gauge-fixed form of the Lagrangian. Completing the square in the $(\partial^a h_{ab})^2$ part, we can rewrite the Lagrangian as
\be
{\mathcal L}_{GR} = \frac{1}{2}  (\partial_a h_{bc})^2 - \frac{1}{4}  (\partial_a h)^2 - ( \partial^a ( h_{ab} -\frac{1}{2} \eta_{ab} h ))^2 .
\ee
If we gauge-fix the diffeomorphisms by setting 
\be
 \partial^a ( h_{ab} -\frac{1}{2} \eta_{ab} h )=0,
 \ee
 we get a simple linear combination of the terms containing $\partial^a \partial_a$ only. 
 
\subsection{The case of ${\rm Spin}(7)$-structures}

Let us now consider 3 fields in irreducible representations of ${\rm Spin}(7)$ given by ${\bf 1, 7, 35}$. These are precisely the representations appearing in a tangent vector to the ${\rm GL}(8)$ orbit of Cayley forms. We will refer to these fields as $h, \xi, \tilde{h}$ respectively. The decomposition into irreducibles is now dictated by the ${\rm Spin}(7)$ representation theory. There are again terms involving $p^2$, which are $p^2 h^2, p^2 \xi^2, p^2 \tilde{h}^2$. To determine other possible terms we need to consider the (symmetric) product of two derivatives.
We have
\be
{\bf 8}\otimes_S {\bf 8} = {\bf 1}+{\bf 35},
\ee
which is unchanged from the ${\rm Spin}(8)$ case. The trivial representation here corresponds to $p^2$, and so we only need to consider the ${\bf 35}$ representation. This must couple to the product of two fields from the list $h, \xi, \tilde{h}$. The non-trivial such decompositions are 
\be
{\bf 7}\otimes_S {\bf 7}= {\bf 1}+{\bf 27}, \\ \nonumber
{\bf 7}\otimes {\bf 35} = {\bf 21}+{\bf 35}+{\bf 189}, \\ \nonumber
{\bf 35}\otimes_S {\bf 35} = {\bf 1} +{\bf 27} + {\bf 35} + {\bf 105} +{\bf 168}+{\bf 294}.
\ee
We are looking for every occurrence of the representation ${\bf 35}$ here. We already know that the terms $h p^a p^b \tilde{h}_{ab}, (p^a \tilde{h}_{ab})^2$ are possible. The second line above shows that there is a new term of the type $p^a p^b \xi \tilde{h}$. It is easy to write down this term by noting that the representation ${\bf 7}$ appears in the anti-symmetric part of the tensor product 
\be
{\bf 35}\otimes_A {\bf 35} = {\bf 7} +{\bf 21} + {\bf 35} + {\bf 189} +{\bf 378}.
\ee
We already know that the best way to describe a field in representation ${\bf 7}$ is by using a field in $\Lambda^2_7$. Thus, let us introduce an object 
\be
\xi_{ab} \in \Lambda^2_7.
\ee
We can then construct the term coupling $\xi_{ab}, \tilde{h}_{ab}$ as $p_a p^c h_{cb} \xi^{ab}$. 

Going back to the notation that involves partial derivatives, it is now clear that there are just two terms that can be constructed from $\xi_{ab}$. These can be written as $\xi^{ab} \partial^c \partial_c \xi_{ab}$ and $\partial_b \tilde{h}^{ba} \partial^c \xi_{ca}$. Note that we can also write the second term as $\partial_b h^{ba} \partial^c \xi_{ca}$, because the trace part of $h_{ab}$ does not couple to $\xi_{ab}$. 
As before, we now write a general linear combination of all the possible terms, with arbitrary coefficients:
\be\label{L-xi}
{\mathcal L}= \frac{\rho}{2} (\partial_a h_{bc})^2 + \frac{\alpha}{2} ( \partial_a h)^2 - \beta h \partial^a \partial^b h_{ab} - \gamma (\partial^a h_{ab})^2 \\ \nonumber
+\frac{\lambda}{2} ( \partial_a \xi_{bc} )^2- \mu \partial_b h^{ba} \partial^c \xi_{ca}.
\ee
As before the notation $(t_{a_1 \ldots a_k})^2$ for any tensor means the complete contraction of two copies of this tensor. 
It will be convenient to put an arbitrary coefficient $\rho$ also in front of the first term. This will allow us to write the general diffeomorphism-invariant Lagrangian as a linear combination of two separately invariant Lagrangians. 

\subsection{Some identities}

Let us now explain why the term $(\partial^a \xi_{ab})^2$ is not added to the Lagrangian (\ref{L-xi}). The representation theory tells us that there is no representation ${\bf 35}$ in the decomposition ${\bf 7}\otimes_S {\bf 7}$, and so this term must be a multiple of $\xi^{bc} \partial^a \partial_a \xi_{bc}$. Let us confirm that. Using the fact that $\xi\in \Lambda^2_7$ we have
\be
\xi_{ab} = - \frac{1}{6} \Phi_{ab}{}^{pq} \xi_{pq},
\ee
and so
\be
\xi_b{}^a \xi_{ca} = \frac{1}{36}\left( - 28 \xi_b{}^a \xi_{ca}  + 8 g_{bc} (\xi_{pq})^2\right).
\ee
From this we get
\be
\xi_b{}^a \xi_{ca} =\frac{1}{8}  g_{bc} (\xi_{pq})^2.
\ee
This explains why the term $(\partial^a \xi_{ab})^2$ is already contained in the $\xi^{bc} \partial^a \partial_a \xi_{bc}$ term in the Lagrangian (\ref{L-xi}) and does not need to be added as a separate term. 

\subsection{Gauge-fixing}

We note that there is a gauge in which the general Lagrangian is given by a sum of terms only involving the Laplacians. Indeed, for $\gamma\not=0$ we can rewrite the Lagrangian as
\be\label{elliptic-gauge}
{\mathcal L}= \frac{\rho}{2} ( \partial_a h_{bc})^2 + \left( \frac{\alpha}{2} + \frac{\beta^2}{4\gamma}\right) ( \partial_a h)^2 
+\left( \frac{\lambda}{2} + \frac{\mu^2}{32\gamma}\right)( \partial_a \xi_{bc})^2 \\ \nonumber
- \gamma (\partial^a (h_{ab} - \frac{\beta}{2\gamma} \eta_{ab} h + \frac{\mu}{2\gamma} \xi_{ab}))^2.
\ee

\subsection{The transformation properties under diffeomorphisms}

To determine the diffeomorphism transformation rules for all the fields we recall that $h_{ab}$ and $\xi_{ab}$ appear from a certain projection of the perturbation of the 4-form. If we call this perturbation $\phi\in \Lambda^4$, the fact that this 4-form is a tangent vector to the orbit of Cayley 4-forms means that $\phi\in \Lambda^4_{1+7+35}$. Let us define the fields $h_{ab}, \xi_{ab}$ as
\be\label{field-defs}
\tilde{h}_{ab} = \frac{1}{24}( \phi_{(a}{}^{pqr} \Phi_{b)pqr} - \frac{1}{8}\eta_{ab} \phi^{pqrs}\Phi_{pqrs}), \\ \nonumber
 \xi_{ab} = \frac{1}{24} \phi_{[a}{}^{pqr} \Phi_{b]pqr}, \quad h = \frac{1}{168}  \phi^{pqrs}\Phi_{pqrs}.
\ee
We emphasise that here and till the end of this section $\Phi$ is the background, which is assumed to be constant, and $\phi$ is the perturbation. 
A calculation shows that the inverse of this map is the following parametrisation of $\phi_{abcd}$
\be\label{phi-param}
\phi_{abcd} = -4 ( h_{[a}{}^{p} +  \frac{1}{4} \xi_{[a}{}^{p}) \Phi_{bcd] p},
\ee
and this formula explains the choice of prefactors in (\ref{field-defs}). We note that
\be
\frac{1}{96} \phi^{abcd} \phi_{abcd} = h^{ab} h_{ab} + \frac{3}{4} h^2 + \frac{1}{4} \xi^{ab} \xi_{ab},
\ee
where it is used that $\xi_{ab}\in \Lambda^2_7$.

Under diffeomorphisms 
\be
\delta \phi = i_\xi d \Phi + d i_\xi \Phi.
\ee
We assume that the background 4-form $\Phi$ is closed (in fact constant), so that there is only the second term. Then
\be\label{lin-diffeo}
\delta \phi_{abcd} = - 4 \partial_{[a} \xi^p \Phi_{bcd] p},
\ee
and so $(1/4) \delta \xi_{ab} = \pi_7(\partial_{[a} \xi_{b]})$ giving 
\be
\delta h_{ab} =  \partial_{(a} \xi_{b)}, \qquad \delta \xi_{ab} = \partial_{[a} \xi_{b]} - \frac{1}{2} \Phi_{ab}{}^{pq} \partial_p \xi_q.
\ee

\subsection{Determining the diffeomorphism-invariant Lagrangian}

The variation of the Lagrangian (\ref{L-xi}), modulo surface terms, is given by 
\be
\delta {\mathcal L} = (  \rho - \gamma - \frac{\mu}{2}) \partial^a h_{ab} \partial^2 \xi^b + (- \beta + \gamma - \frac{\mu}{2}) \partial^a \partial^b h_{ab} (\partial\xi) 
\\ \nonumber
-( \alpha +\beta) \partial^2 h (\partial\xi)
+ ( 4\lambda - \frac{\mu}{2}) \partial^a \xi_{ab} \partial^2 \xi^b .
\ee
Here $\partial^2 = \partial^a \partial_a$ and $(\partial \xi)=\partial^a \xi_a$.
We have used the fact that $\xi_{ab}$ is in $\Lambda^2_7$, and so $(1/2) \Phi_{ab}{}^{pq} \xi_{pq} = - 3\xi_{ab}$. Setting to zero the coefficients in front of the independent parts we get a system of equations. The solution depends on two of the parameters, for which we can take $\rho, \mu$. Then
\be\label{parameters}
\alpha = -\rho+\mu, \quad \beta = \rho-\mu, \quad \gamma = \rho -\frac{\mu}{2}, \quad \lambda= \frac{\mu}{8}.
\ee
It is clear that the resulting diffeomorphism-invariant Lagrangian is the sum of two separately invariant terms
\be\label{L-diff}
{\mathcal L}=  \rho {\mathcal L}_{GR} +\mu {\mathcal L}',
\ee
where
\be
{\mathcal L}' =  \frac{1}{2} ( \partial_a h)^2 + h \partial^a \partial^b h_{ab} +\frac{1}{2} (\partial^a h_{ab})^2
+\frac{1}{16}( \partial_c \xi_{ab})^2 -  \partial_b h^{ba} \partial^c \xi_{ca}.
\ee
We thus observe that the linearised action in the case of ${\rm Spin}(7)$-structures is not unique. There are two linearly independent such actions, and the general action is given by their linear combination. One of the parameters can always be absorbed into the perturbation of the 4-form field, but the other parameter remains. 

It is interesting to remark that the story we described parallels precisely the story one finds in the case of ${\rm SU}(2)$ structures in four dimensions, see \cite{Bhoja:2024xbe} section 6. In four dimensions there are also two diffeomorphism-invariant terms that can be written down for perturbations of an ${\rm SU}(2)$ structure. The main difference with four dimensions is that in that case there is an additional symmetry that can be invoked, namely ${\rm SU}(2)$ gauge transformations, that allows to eliminate one of the two independent terms. In the case of four dimensions the origins of this extra symmetry are in the fact that $\mathfrak{so}(4)$ splits as $\mathfrak{so}(4)=\mathfrak{g}\oplus \mathfrak{g}^\perp$, and $\mathfrak{g}^\perp$ is also a Lie algebra. It is the requirement of gauge invariance with respect to $\mathfrak{g}^\perp$ that eliminates one of the two possible terms in the case of 4D. In the present case of eight dimensions $\mathfrak{g}^\perp$ is not a Lie algebra, and no similar requirement of gauge invariance is possible. However, it is not impossible that for some special value of the ratio $\mu/\rho$ an extra gauge symmetry arises in the theory described by (\ref{L-diff}), and this selects a preferred member of the family of theories. At the moment of writing this remark we do not know whether this is the case, but this possibility is under investigation.

\section{Intrinsic torsion}

We now proceed to our construction of the non-linear theories completing the linear story described above. The purpose of this section is to recall the definition of the intrinsic torsion of a ${\rm Spin}(7)$-structure and establish some facts that are necessary for the following. 

\subsection{Characterisation of the intrinsic torsion}

We start with the following proposition, whose proof can also be found in \cite{Spiro-Spin7}.
\begin{proposition} The intrinsic torsion of a ${\rm Spin}(7)$-structure, measured by $\nabla_a \Phi_{ijkl}$, where $\nabla_a$ is the metric-compatible covariant derivative, takes values in $\Lambda^1\otimes \Lambda^4_7$. Using the isomorphism $\Lambda^2_7 \sim \Lambda^4_7$ provided by the operator $K$, see (\ref{operator-K}), the intrinsic torsion can be parametrised by an object in $\Lambda^1\otimes \Lambda^2_7$. Explicitly,
\be\label{torsion}
\nabla_a \Phi_{ijkl} = T_{a; ip} \Phi^p{}_{jkl} - T_{a; jp} \Phi^p{}_{kli}+T_{a; kp} \Phi^p{}_{lij}- T_{a; lp} \Phi^p{}_{ijk}, \\ \nonumber
 T_{a;ij}\in \Lambda^1\otimes \Lambda^2_7.
\ee
\end{proposition}
We remark that we use the same notation for the intrinsic torsion $T_{a;ij}$ as in \cite{Spiro-Spin7}. The semi-colon here should not be confused with the symbol denoting the covariant derivative, which is standard in some physics literature. We never use this notation for the covariant derivative in this paper, and thus we hope that no confusion arises. 
\begin{proof}
The proof of this proposition consists in showing that the projections of $\nabla_a \Phi_{ijkl} \in \Lambda^1\otimes \Lambda^4$ to all other irreducible components of $\Lambda^4$ apart from $\Lambda^4_7$ vanish. It is given in \cite{Spiro-Spin7}, and similar computations in the case of $G_2$ structures are spelled out in \cite{Spiro-G2}. We spell out an alternative, completely explicit proof, which is made possible by our knowledge of the projections to $\Lambda^4_{35+1}$ and the expression (\ref{projector-27}) for the projector to $\Lambda^4_{27}$. The projection to $\Lambda^4_{35+1}$ is obtained by computing 
\be
2\Phi_{(i}{}^{pqr} \nabla_{|a|} \Phi_{j)pqr} = \nabla_a \Phi_{i}{}^{pqr}\Phi_{jpqr} = 42 \nabla_a g_{ij} =0.
\ee
For the projection on $\Lambda^4_{27}$ the computation is a bit more involved. First, we need some identities. We have, on one hand
\be\nonumber
\nabla_p ( \Phi_{[ij}{}^{ab} \Phi_{kl]}{}^{cd} \Phi_{abcd}) =  \\ \nonumber
\Phi_{[kl}{}^{cd} \Phi_{|abcd}  \nabla_{p|} ( \Phi_{ij]}{}^{ab}) + \Phi_{[ij}{}^{ab} \Phi_{|abcd} \nabla_{p|} ( \Phi_{kl]}{}^{cd} )
+ \Phi_{[ij}{}^{ab} \Phi_{kl]}{}^{cd} \nabla_p \Phi_{abcd}=\\ \nonumber
12 \delta_{[k}^a \delta_{l}^b \nabla_{|p|} \Phi_{ij]ab} - 4 \Phi_{[kl}{}^{ab} \nabla_{|p|} \Phi_{ij]ab}  + 12 \delta_{[i}^a \delta_{j}^b \nabla_{|p|} \Phi_{kl]ab} - 4 \Phi_{[ij}{}^{ab} \nabla_{|p|} \Phi_{kl]ab} 
\\ \nonumber + \Phi_{[ij}{}^{ab} \Phi_{kl]}{}^{cd} \nabla_p \Phi_{abcd}
= 24 \nabla_{p} \Phi_{ijkl} - 8 \Phi_{[ij}{}^{ab} \nabla_{|p|} \Phi_{kl]ab} 
+ \Phi_{[ij}{}^{ab} \Phi_{kl]}{}^{cd} \nabla_p \Phi_{abcd}.
\ee
On the other hand
\be
\nabla_p ( \Phi_{[ij}{}^{ab} \Phi_{kl]}{}^{cd} \Phi_{abcd}) = 28 \nabla_p \Phi_{ijkl}.
\ee
Thus, we have
\be
\Phi_{[ij}{}^{ab} \Phi_{kl]}{}^{cd} \nabla_p \Phi_{abcd} =4 \nabla_p \Phi_{ijkl}+ 8 \Phi_{[ij}{}^{ab} \nabla_{|p|} \Phi_{kl]ab} .
\ee
We also have 
\be
\nabla_p (\Phi_{[ij}{}^{ab} \Phi_{kl]ab}) = 2 \Phi_{[ij}{}^{ab} ( \nabla_{|p|} \Phi_{kl]ab}).
\ee
On the other hand, 
\be
\nabla_p (\Phi_{[ij}{}^{ab} \Phi_{kl]ab}) = - 4 \nabla_p \Phi_{ijkl},
\ee
and so
\be
\Phi_{[ij}{}^{ab} ( \nabla_{|p|} \Phi_{kl]ab}) = - 2\nabla_p \Phi_{ijkl}, \\ \nonumber
\Phi_{[ij}{}^{ab} \Phi_{kl]}{}^{cd} \nabla_p \Phi_{abcd} =-12 \nabla_p \Phi_{ijkl}.
\ee
Using (\ref{projector-27}), these identities, as well as $\Phi^{abcd} \nabla_p \Phi_{abcd} =0$, it is easy to see that
\be
\pi_{27} (\nabla_p \Phi_{ijkl}) = 0.
\ee
Finally, to establish (\ref{torsion}) we just need to recall that a general element of $\Lambda^4_7$ can be parametrised as $K(\beta), \beta\in \Lambda^2_7$, where $K:\Lambda^2_7\to \Lambda^4_7$ is the map introduced in (\ref{operator-K}). We thus have
\be
\nabla_a \Phi_{ijkl} = -4 T_{a; [i|p|} \Phi_{jkl]}{}^p,
\ee
where $T_{a;ij}\in \Lambda^1\otimes \Lambda^2_7$. This is precisely the formula (\ref{torsion}). 
\end{proof}

\subsection{Parametrisation by the torsion 3-form}

As is known, see e.g. \cite{Friedrich} Example 3.4., the spaces $\Lambda^1\otimes \Lambda^2_7$ and $\Lambda^3$ are isomorphic. We can make this isomorphism explicit, in one direction, by parametrising the intrinsic torsion $T_{a;ij}$ as follows
\be\label{torsion-parametrised}
T_{a;ij} = \pi_7(T_{aij}) = \frac{1}{4} T_{aij} - \frac{1}{8} \Phi_{ij}{}^{kl} T_{akl}, \qquad T_{aij}\in \Lambda^3.
\ee
An explicit relation in the other direction is
\be
T_{aij} = \frac{4}{3} T_{a;ij} + 4 T_{[a;ij]} +  T_{[a;kl]} \Phi_{ij}{}^{kl} + \frac{2}{9} T_{k;lm} \Phi^{klm}{}_{[i} g_{j]a}.
\ee

Using this parametrisation, we can rewrite (\ref{torsion}) in terms of the torsion 3-form.
\begin{proposition} 
In the parametrisation of the intrinsic torsion by a torsion 3-form, we have
\be\label{torsion-3-form}
\nabla_a \Phi_{ijkl} = T_{aip} \Phi_{pjkl} - T_{ajp} \Phi_{pkli}+T_{akp} \Phi_{plij}- T_{alp} \Phi_{pijk}, \\ \nonumber
T_{aij}\in \Lambda^3.
\ee
\end{proposition}
A proof is by explicit verification, substituting (\ref{torsion-parametrised}) into (\ref{torsion}). Note that this is the same formula for the covariant derivative of the basic 4-form, but now with the torsion 3-form instead of the object $T_{a;ij}\in \Lambda^1\otimes\Lambda^2_7$. We would like to emphasise that the approach of this paper to a large extent depends on the existence of the formula (\ref{torsion-3-form}). 

\subsection{Connection with skew-symmetric torsion}

As is known, see \cite{Ivanov}, \cite{Friedrich}, any ${\rm Spin}(7)$-structure on an 8-dimensional manifold admits a unique connection with totally skew-symmetric torsion. Such a connection is given by
\be
\tilde{\nabla}_a X_i = \nabla_a X_i - T_{aip} X_p.
\ee
It is then clear that the relation (\ref{torsion-3-form}) can be interpreted as the statement that the 4-form is parallel with respect to $\tilde{\nabla}$
\be
\tilde{\nabla}_a \Phi_{ijkl} =0.
\ee
The known existence of a unique connection with totally skew-symmetric torsion thus gives an alternative justification why the formula (\ref{torsion-3-form}) must exist. 

\subsection{Torsion 3-form from the exterior derivative of the Cayley form}

\begin{proposition} The torsion 3-form is completely determined by the exterior derivative $d\Phi$. Explicitly, we have
\be\label{T-d-Phi}
T=  \frac{1}{2} J_3^{-1}(\star (d\Phi)) ,
\ee
where $J_3$ is the operator in 3-forms introduced in (\ref{J3}), and $\star (d\Phi)$ is the Hodge dual of $d\Phi$.
\end{proposition}
\begin{proof}
On one hand, we have
\be
\star (d\Phi)_{mnr} = \frac{1}{5!} \epsilon_{mnr}{}^{aijkl} 5 \partial_{[a} \Phi_{ijkl]} = \frac{1}{4!} \epsilon_{mnr}{}^{aijkl} \partial_{a} \Phi_{ijkl} 
\ee
On the other hand, substituting here the right-hand-side of (\ref{torsion-3-form}) we have
\be
 \frac{1}{4!} \epsilon_{mnr}{}^{aijkl} \partial_{a} \Phi_{ijkl} = \frac{1}{6} \epsilon_{mnr}{}^{aijkl} T_{aip} \Phi_{pjkl}.
 \ee
Now, using (\ref{epsilon-phi-3}) we get
\be
\epsilon_{mnr}{}^{aijkl} T_{aip} \Phi_{pjkl}=  \\ \nonumber
6 (\Phi_{mn}{}^{pq} T_{rpq}+ \Phi_{nr}{}^{pq} T_{mpq}+\Phi_{rm}{}^{pq} T_{npq}) =
12 J_3(T)_{mnr}.
\ee
This means we have
\be\label{torsion-dT}
\star (d\Phi)_{mnr} =  2 J_3(T)_{mnr}.
\ee
Now, the operator $J_3$ is invertible, with inverse given by (\ref{J3-inv}). 
This proves the proposition.
\end{proof}

\section{Riemann curvature identities}

Having described the intrinsic torsion and its relation with the covariant and exterior derivatives of the Cayley form, we can obtain very useful characterisations of (some parts of) the Riemann curvature. This material is standard, see for example Theorem 2.10 in \cite{Dwivedi}. The difference in our treatment is that we use the parametrisation of the torsion by a 3-form. 

\subsection{Irreducible components of the Riemann tensor}

This material is well-known, see e.g. \cite{Dwivedi}. 
The Riemann tensor is an object with values in ${\rm Sym}^2(\Lambda^2)$, with $\Lambda^4$ removed. Given that $\Lambda^2=\Lambda^2_{7}\oplus\Lambda^2_{21}$, it is easy to compute the decomposition of ${\rm Sym}^2(\Lambda^2)$ into irreducibles using the well-known facts about the tensor products of irreducible representations of ${\rm Spin}(7)$. We denote representations using the corresponding dimension written in bold face. We need the following tensor product decompositions:
\be
{\bf 7}\otimes_S {\bf 7} = {\bf 1} \oplus {\bf 27}, \\ \nonumber
{\bf 7}\otimes {\bf 21} = {\bf 105}\oplus {\bf 35}\oplus {\bf 7}, \\ \nonumber
{\bf 21}\otimes_S {\bf 21} = {\bf 1}\oplus {\bf 27} \oplus {\bf 35}\oplus {\bf 168}.
\ee
Taking into account that
\be
\Lambda^4 = {\bf 1}\oplus{\bf 7}\oplus {\bf 27} \oplus {\bf 35},
\ee
we see that Riemann curvature gets decomposed into the following irreducible components
\be
\text{Riemann} = {\bf 1}\oplus {\bf 27} \oplus {\bf 35} \oplus {\bf 105} \oplus {\bf 168}.
\ee
Of these the Ricci part is
\be
\text{Ricci}=  {\bf 1}\oplus {\bf 35}, 
\ee
and the Weyl part is
\be
\text{Weyl} = {\bf 27}  \oplus {\bf 105} \oplus {\bf 168}.
\ee
Our next task is to characterise which parts of the Riemann curvature can be extracted from the intrinsic torsion. 

\subsection{Part of Riemann curvature from the torsion}

We now take the commutator of two covariant derivatives applied to the basic 4-form to get
\be
4R_{ab[i}{}^p \Phi_{|p|jkl]}  = 2 \nabla_{[a}\nabla_{b]} \Phi_{ijkl} =
 8\nabla_{[a}( T_{b][i|p|} \Phi^{p}{}_{jkl]} ).
\ee
Applying the product rule and using (\ref{torsion-3-form}) one more time we get
\be\label{two-nabla}
4R_{ab[i}{}^p \Phi_{|p|jkl]}  = 4\nabla_{a}( T_{b[i|p|}) \Phi^{p}{}_{jkl]} - 4\nabla_{b}( T_{a[i|p|}) \Phi^{p}{}_{jkl]} \\ \nonumber
+ 4T_{a[i}{}^p T_{|bp|}{}^q \Phi_{jkl]q} - 4T_{b[i}{}^p T_{|ap|}{}^q \Phi_{jkl]q} .
\ee
This is what is known as the Bianchi identity in the literature, see e.g. \cite{Spiro-Spin7} Theorem 4.2. 

\subsection{Identity for the divergence of the torsion 3-form}

Before we proceed any further, a useful consequence of this identity is obtained by multiplying it with $\epsilon^{mn ab ijkl}$, and using (\ref{epsilon-phi-3}). On the left-hand side we get identically zero, by properties of the Riemann curvature. The right-hand side is non-trivial and so we get
\be
2\nabla_{a} T_{bi[m} \Phi_{n]}{}^{abi}  + \nabla^a T_{abi} \Phi^{bi }{}_{mn}    \\ \nonumber
+2 T_{ai}{}^p T_{bp[m} \Phi_{n]}{}^{abi} 
- 2T_{a}{}^{pq} T_{bpq} \Phi_{mn}{}^{ab}=0.
\ee
We would now like to extract from here the divergence $\nabla^a T_{amn}$ of the torsion 3-form in terms of other quantities. Applying to this expression $(1/4)(\mathbb{I}+J_2)$, we get
\be\label{divergence-torsion}
\nabla^a T_{amn} =  \frac{1}{2} \Phi_{[m}{}^{abc} \nabla_{n]} T_{abc}- \frac{1}{2} \Phi_{[m}{}^{abc} \nabla_{|a|} T_{n]bc}   \\ \nonumber
- \Phi_{[m}{}^{abc} T_{n]a}{}^p T_{bcp}.
\ee
We can rewrite this in a different form, by applying the projection to $\Lambda^2_7$. We get
\be\label{T-divergence}
\nabla^a T_{amn} - \frac{1}{2} \Phi_{mn}{}^{pq} \nabla^a T_{apq} =   \\ \nonumber
 \frac{1}{2} \Phi_{[m}{}^{abc} \nabla_{n]} T_{abc} - \frac{3}{2} \Phi_{[m}{}^{abc} \nabla_{|a|} T_{n]bc} - 2\Phi_{[m}{}^{abc} T_{n]a}{}^p T_{bcp}.
\ee
It is useful to rewrite this as the divergence of the original torsion. Using (\ref{torsion-3-form}) we have
\be
4 \nabla^a T_{a;mn} = \nabla^a T_{amn} - \frac{1}{2} \Phi_{mn}{}^{pq} \nabla^a T_{apq} - \Phi_{[m}{}^{abc} T_{n]a}{}^p T_{bcp},
\ee
and thus
\be\label{T-divergence-identity}
4 \nabla^a T_{a;mn}= \frac{1}{2} \Phi_{[m}{}^{abc} \nabla_{n]} T_{abc} - \frac{3}{2} \Phi_{[m}{}^{abc} \nabla_{|a|} T_{n]bc}   \\ \nonumber
- 3\Phi_{[m}{}^{abc} T_{n]a}{}^p T_{bcp}.
\ee

\subsection{Component of the Riemann curvature}

Thinking about $R_{abcd}$ as an object in $\Lambda^2\otimes_S \Lambda^2$ (with a copy of $\Lambda^4$ removed), and recalling the operator $K$ introduced in (\ref{operator-K}), we see that the object on the left-hand side of (\ref{two-nabla}) is valued in $\Lambda^2\otimes \Lambda^4_7$. We can then apply the inverse operator $K'$ to obtain
\be\label{bianchi}
 R_{abij} - \frac{1}{2} \Phi_{ij}{}^{pq} R_{abpq} = \nabla_a T_{bij} - \nabla_b T_{aij} \\ \nonumber
 - \frac{1}{2} \Phi_{ij}{}^{cd} \nabla_a T_{b cd}+ \frac{1}{2} \Phi_{ij}{}^{cd} \nabla_b T_{a cd}
+ T_{ai}{}^p T_{bjp} - T_{bi}{}^p T_{ajp}- \Phi_{ij}{}^{pq} T_{ap}{}^k T_{bqk}.
\ee
Both sides of this equality can be checked to be in $\Lambda^2_{7}$ with respect to indices $ij$, by applying the projector to $\Lambda^2_{21}$ and seeing that the result is identically zero. This computation makes it obvious that all apart from the ${\bf 168}$ part of the Weyl curvature are determined by the intrinsic torsion. Indeed, all parts but this one come from $\Lambda^2 \otimes \Lambda^2_7$, and this is precisely what the part of the Riemann curvature tensor that the intrinsic torsion determines. 

\subsection{Ricci curvature scalar}

Before we use the facts above to obtain a formula for the Ricci curvature, let us note that there are two different ways to extract the Ricci scalar from here. One is to contract the indices with $g^{ai} g^{bj}$. The other is to contract it with $-(1/6) \Phi^{abij}$. Both of these give
\be\label{ricci-scalar}
R = - \Phi^{abcd} \nabla_a T_{bcd} + T^{abc} T_{abc} + \Phi^{abcd} T_{ab}{}^p T_{cdp}.
\ee
Note that this only depends on the exterior derivative $dT$ of the torsion 3-form.

\subsection{Einstein-Hilbert action}

It will be useful for what follows to compute the Einstein-Hilbert action as a linear combination of torsion squared terms. Denoting the volume form $v_\Phi=v_g$, we have
\be\nonumber
S_{EH}[g] = \int R v_g d^8x = \int \left( - \Phi^{abcd} \nabla_a T_{bcd} + T^{abc} T_{abc} + \Phi^{abcd} T_{ab}{}^p T_{cdp}\right) v_g d^8x.
\ee
Integrating by parts in the first term we have
\be
S_{EH}[g] = \int \left( (\nabla_a \Phi^{abcd})  T_{bcd} + T^{abc} T_{abc} + \Phi^{abcd} T_{ab}{}^p T_{cdp}\right) v_g d^8x.
\ee
We can now use (\ref{torsion-3-form}) to get
\be\label{nabla-phi-contr}
\nabla_a \Phi^{abcd} =  -  3\Phi^{[bc}{}_{pq} T^{d]pq} = - 2J_3(T)^{bcd}.
\ee
This means that the first term is a multiple of $T J_3(T)$. Rewriting also the last term in the same way we have
\be\label{S-EH}
S_{EH}[g] = \int \left(  T^{abc} T_{abc}-\frac{4}{3} J_3(T)^{abc}  T_{abc}  \right) v_g d^8x.
\ee
We will need this result later.

\subsection{Extracting the Ricci curvature}

We can extract the Ricci tensor from (\ref{bianchi}) by multiplying with $\Phi_{c}{}^{bij}$, and applying the Bianchi identity $R_{a[bij]}=0$ to get
\be
- \frac{1}{2} \Phi_{ij}{}^{pq} R_{abpq} \Phi_c{}^{bij} = - 6 R_{ac}.
\ee
Doing the same operations with the right-hand side and we get
\be
R_{ab} = - \nabla^c T_{abc} - \frac{1}{2} \Phi_{b}{}^{ijk} \nabla_{a} T_{ijk}  + \frac{1}{2} \Phi_{b}{}^{ijk} \nabla_{i} T_{ajk} \\ \nonumber
+T_a{}^{pq} T_{bpq} + \Phi_{b}{}^{ijk} T_{ai}{}^p T_{jkp} .
\ee
This is not explicitly symmetric in $ab$, and must therefore become symmetric when $T_{ijk}$ is given by its expression (\ref{T-d-Phi}). And indeed, the anti-symmetric part of the right-hand side vanishes in view of (\ref{divergence-torsion}). Thus, the Ricci curvature is given by
\be\label{Ricci}
R_{ab} =  - \frac{1}{2} \Phi_{(a}{}^{ijk} \nabla_{b)} T_{ijk}  + \frac{1}{2} \Phi_{(a}{}^{ijk} \nabla_{|i|} T_{b)jk} \\ \nonumber
+T_a{}^{pq} T_{bpq} + \Phi_{(a}{}^{ijk} T_{b)i}{}^p T_{jkp} .
\ee
We have now proven the result known since \cite{Fernandez}: when $d\Phi=0$ the metric is Ricci-flat. Indeed, by (\ref{T-d-Phi}) $d\Phi=0$ implies $T=0$, which in turn gives $R_{ab}=0$ by (\ref{Ricci}). Note that, unlike the Ricci scalar (\ref{ricci-scalar}), the Ricci tensor depends on the full covariant derivative of the torsion 3-form. 

\section{The action}

This section is central to the whole paper. We consider a two-parameter family of action functionals of the Cayley form $\Phi$ and an auxiliary 3-form $C$. The Lagrangians we consider are first-order in derivatives, and the Euler-Lagrange equations for the auxiliary field $C$ are algebraic. There is a member in the family of actions for which the Euler-Lagrange equation for $C$ equates it with the intrinsic torsion $T$ of the ${\rm Spin}(7)$-structure. For any member of the family of action functionals, once the Euler-Lagrange equation for $C$ is solved and $C$ is obtained in terms for the derivative of $\Phi$, this value of $C$ can be substituted back into the Lagrangian, resulting in a second-order in derivatives Lagrangian that depends solely on $\Phi$. As we will see, the Lagrangians arising this way are basically a linear combination of the two invariants that can be constructed from the intrinsic torsion of $\Phi$. We will analyse the actions arising this way, and derive the arising Euler-Lagrange equations for $\Phi$. 

There are some limitations to our construction. First, our motivation is to mimic what happens in Plebanski formalism that gives an efficient description of 4D General Relativity \cite{Bhoja:2024xbe}. Plebanski action is similarly a functional that depends on an ${\rm SU}(2)$ structure (encoded into a triple of 2-forms), and an auxiliary (connection) field. The action is first-order in derivatives and quadratic in the auxiliary field. After the connection is solved for from its Euler-Lagrange equations and substituted back into action action, one obtains a second-order action, which is just one of the two possible invariants that can be built from the intrinsic torsion of the ${\rm SU}(2)$ structure. Our construction generalises all this to the case of ${\rm Spin}(7)$ structures. Motivated by this example of Plebanski formalism, we do not allow more complicated than quadratic dependence of the action on the auxiliary field $C$, as we would like to retain the possibility to solve for $C$ explicitly. It is clear that more involved first-order actions depending on $\Phi, C$ can be constructed, with more complicated dependence on $C$ than quadratic, but this is not pursued in the present paper.

\subsection{A two-parameter family of action functionals}

The action we want to construct is a functional of $\Phi\in \Lambda^4$ and $C\in \Lambda^3$. It will contain a term imposing the constraints that guarantee that $\Phi$ is of algebraic type of the Cayley form. We will never need to specify what these constraints are, as we will only need their consequences. Given $\Phi,C$ there is a natural top form that can be constructed, which is $\Phi\wedge dC$. We take the integral of this to be our 'kinetic', i.e. containing derivatives term. Lagrangians of this type are well-known in the context of topological field theories. Thus, the theory with the action $\int \Phi \wedge dC$, with no additional terms, is a topological field theory known as (Abelian) BF theory. Our Lagrangian, however, contains other terms which render the theory non-topological. Apart from the terms imposing the constraints on $\Phi$, we also want the Lagrangian to contain terms quadratic in $C$, such that the variation of the action with respect to $C$ gives a set of linear equations for $C$. 

The representation theoretic fact $\Lambda^3=\Lambda^3_8\oplus\Lambda^3_{48}$ implies that there are two linearly independent quadratic invariants that can be constructed from a 3-form $C$. A computation gives
\be\nonumber
|C_8|^2=\pi_8(C)_{abc} C^{abc} = \frac{1}{7} (C_{abc})^2 - \frac{3}{14} \Phi^{abcd} C_{ab}{}^p C_{cdp} = \frac{1}{7}(C^2 - C J_3(C)), \\ \nonumber
|C_{48}|^2=\pi_{48}(C)_{abc} C^{abc} = \frac{6}{7} (C_{abc})^2 + \frac{3}{14} \Phi^{abcd} C_{ab}{}^p C_{cdp} =\frac{1}{7}(6 C^2 + C J_3(C)).
\ee
This shows that the two linearly independent quadratic invariants constructed from $C$ can be taken to be $(C_{abc})^2$ and $\Phi^{abcd} C_{ab}{}^p C_{cdp}$. The coefficient in front of one of these can always be chosen as desired by rescaling the $C$ field. This leads us to consider the following family of action functionals
\be\label{action-section-kappa}
S[\Phi,C] = \int \Phi \wedge (dC -6  C\wedge_\Phi C) + \frac{\kappa}{6} (C)^2 v_\Phi + \frac{\lambda}{6} v_\Phi + \text{constr.}
\ee
The choice of coefficients here will be convenient for what follows. The constant $\lambda$ is a 'cosmological constant' term that can be set to zero if desired. The object $C\wedge_\Phi C$ is the 4-form
\be
C\wedge_\Phi C := \frac{1}{4!} (C\wedge_\Phi C)_{ijkl}  \,dx^i \wedge dx^j \wedge dx^k\wedge dx^l, \\ \nonumber  (C\wedge_\Phi C)_{ijkl}= g^{pq} C_{[ij|p|} C_{kl]q}
\ee
and $v_\Phi \equiv v_g = (1/14) \Phi \wedge \Phi$ is the volume form. Written in index notation the action becomes
\be
S[\Phi,C] =\frac{1}{3!}  \int \Big( \frac{1}{4!} \tilde{\epsilon}^{ijklabcd}  \Phi_{ijkl} (\partial_a C_{bcd} - \frac{3}{2} g^{pq} C_{abp} C_{cdq} ) \\ \nonumber
+ \kappa g^{ai} g^{bj} g^{ck} C_{abc} C_{ijk} v_g + \lambda v_g \Big) d^8x.
\ee
The constraint terms are omitted for brevity. The object $\tilde{\epsilon}^{ijklabcd}$ is the density weight one totally anti-symmetric tensor. This exists on any orientable manifold, and does not need a metric for its definition. We emphasise that $\tilde{\epsilon}^{ijklabcd}$ is independent of any metric, in particular the metric defined by $\Phi$, to make it clear that this tensor is not subject to variation when Euler-Lagrange equations are derived below. Using the self-duality (\ref{self-duality}) of $\Phi$ we can see that the two scalars added to the Lagrangian are indeed $(C_{abc})^2$ and $\Phi^{abcd} C_{ab}{}^p C_{cdp}$.

\subsection{The variation with respect to the 3-form}

The variation of the action with respect to $C$ is given by
\be
\delta_C S= \frac{1}{3!} \int v_g \Big(-   \frac{1}{4!} \epsilon^{ijklabcd}  \partial_{a} \Phi_{ijkl} -3 \Phi^{aecd}  C_{ae}{}^b \\ \nonumber
+ 2\kappa C^{bcd} \Big) \delta C_{bcd} d^8x,
\ee
where $v_g$ is the volume form for $g$, and we used the self-duality of the basic 4-form in the second term. The resulting Euler-Lagrange equation is therefore
\be
  \frac{1}{4!} \epsilon_{bcd}{}^{aijkl}  \partial_{a} \Phi_{ijkl} - 2 J_3(C)_{bcd}+ 2\kappa C_{bcd} =0.
\ee
When $\kappa=0$, comparing to (\ref{torsion-dT}), we see that $C=T$. The coefficient in front of the second term in the action was selected so that this happens. In general we have
\be
J_3(T) =  J_3(C) - \kappa C.
\ee
For a general $\kappa$ this relation can be inverted
\be\label{C-T}
C= \frac{6 T +\kappa J_3(T)}{6 - (5+\kappa)\kappa},
\ee
which shows that $\kappa=1, -6$ are the values when the relation cannot be inverted. Note that these are also the eigenvalues of $J_3$. We are particularly interested in the case when $\kappa=0$, where $C=T$, and $\kappa=-2$ where 
\be\label{C-GR}
C= \frac{1}{2} T - \frac{1}{6} J_3(T).
\ee

\subsection{Variation of the metric with respect to the 4-form}

To vary the action with respect to the 4-form, we need a formula for the variation of $g^{ij}$ with respect to the 4-form $\Phi_{ijkl}$. This is standard, see e.g. \cite{Spiro-Spin7}. We provide the full derivation in our notations for convenience. The best way to obtain a relation between the variations is to consider a variation of the metric, thought of as an ${\rm GL}(8,\R)$ transformation. As we have already discussed in (\ref{h-delta-phi}), such a transformation effected by a symmetric $8\times 8$ matrix $h_{ij}$ induces a change in the basic 4-form given by 
\be
K(h)_{ijkl} = 4 h_{[i|p|} \Phi^p{}_{jkl]}.
\ee
It will be more convenient, however, to consider the variation of $\Phi^{ijkl}$. We have
\be\label{delta-phi}
\delta\Phi_{ijkl} = 4\alpha \delta g_{[i|p|} \Phi^p{}_{jkl]}.
\ee
The coefficient of proportionality $\alpha$ should be fixable by taking the variation of any of the algebraic relations satisfied by $\Phi$. For example we have
\be
\Phi_{abcd}\Phi_{ijkl} g^{ia} g^{jb} g^{kc} g^{ld} = 336.
\ee
Varying this gives
\be
2 \delta\Phi_{ijkl} \Phi^{ijkl} + 4\cdot 42 \delta g^{ia} g_{ia} =0,
\ee
where we used (\ref{phi-triple-contr}). Using (\ref{delta-phi}) we have
\be
4 \alpha  \delta g^{[i|p|} \Phi_p{}^{jkl]} \Phi_{ijkl} + 2\cdot 42 \delta g_{ia} g^{ia} =0.
\ee
Using (\ref{phi-triple-contr}) again this becomes
\be
2\alpha \delta g_{ij} g^{ij} +  \delta g^{ij} g_{ij}=0,
\ee
which shows that $\alpha=1/2$. Thus, we have
\be\label{delta-phi*}
\delta\Phi_{ijkl} = 2 \delta g_{[i|p|} \Phi^p{}_{jkl]}.
\ee
As a check of consistency of these expressions, we also compute
\be
\delta ( \Phi^{ijkl}) = \delta ( g^{ia} g^{jb} g^{kc} g^{ld} \Phi_{abcd}) =  \\ \nonumber
g^{ia} g^{jb} g^{kc} g^{ld} \delta\Phi_{abcd} + 4 \delta^{[i|a} g^{jb} g^{kc} g^{l]d}  \Phi_{abcd}
= \\ \nonumber
2 \delta g^{[i|p|} \Phi_p{}^{jkl]} - 4  \delta g^{[i|p|} \Phi_p{}^{jkl]}=-2 \delta g^{[i|p|} \Phi_p{}^{jkl]} = - g^{ia} g^{jb} g^{kc} g^{ld} \delta\Phi_{abcd} .
\ee
This is analogous to the relation that we have for the metric 
\be
\delta g^{ij} = - g^{ia} g^{jb} \delta g_{ab}.
\ee

We now extract $\delta g_{ij}$ in terms of $\delta \Phi_{ijkl}$. To do so we multiply the above expression by $\Phi^a{}_{jkl}$. We get
\be
\delta\Phi_{(i|jkl|} \Phi_{a)}{}^{jkl} = 12 \delta g_{ia} + 9 \delta g_{pq} g^{pq} g_{ia}.
\ee
One more contraction gives 
\be
\delta g_{pq} g^{pq} = \frac{1}{84} \delta\Phi_{ijkl} \Phi^{ijkl},
\ee
and so 
\be
\delta g_{ij} = \frac{1}{12} ( \delta\Phi_{(i|pqr|} \Phi_{j)}{}^{pqr} - \frac{3}{28} g_{ij} \delta\Phi_{pqrs} \Phi^{pqrs}).
\ee
Because the variation of the 4-form with all upper indices is given by minus the variation of the form with the lower indices, and the same is true for the metric variation, we can also write
\be\label{delta-g}
\delta g^{ij} = \frac{1}{12} ( \delta\Phi^{(i|pqr|} \Phi^{j)}{}_{pqr} - \frac{3}{28} g^{ij} \delta\Phi^{pqrs} \Phi_{pqrs}),
\ee
which is the form of the relation that will be used later. 

\subsection{Variation of the action with respect to the 4-form}

We now derive the other half of the Euler-Lagrange equations. We first rewrite the action in terms of $\Phi^{abcd}$
\be\label{action-for-variation}
S[\Phi,C] =\frac{1}{3!}  \int \Big( \Phi^{abcd}   (\partial_a C_{bcd} - \frac{3}{2} g^{pq} C_{abp} C_{cdq} ) \\ \nonumber
+ \kappa g^{ap} g^{bq} g^{cr} C_{abc} C_{pqr}  + \lambda  \Big) v_g d^8x.
\ee
and then vary with respect to $\Phi^{abcd}$. We have\footnote{The first version of this paper, as well as the journal published version, contained a variation mistake in that the dependence of the $\kappa$-term on the inverse metric was not properly taken into account. The author is grateful to Sam Close for spotting this mistake.}
\be
\delta_\Phi S[\Phi,T] =\frac{1}{3!}  \int  \Big[ \delta \Phi^{abcd} (\partial_aC_{bcd} - \frac{3}{2} g^{pq} C_{abp} C_{cdq})  \\ \nonumber
-  \Phi^{abcd} \frac{3}{2} \delta g^{pq} C_{abp} C_{cdq} + 3\kappa g^{ac} g^{bd} \delta g^{pq} C_{pab} C_{qcd} \\ \nonumber
- \frac{1}{2} \delta g^{pq} g_{pq} \Big( \Phi^{abcd} (\partial_aC_{bcd} - \frac{3}{2} g^{pq} C_{abp} C_{cdq})  
+\kappa (C_{abc})^2 + \lambda \Big) \Big] v_g d^8 x.
\ee
The terms containing $\delta g^{pq} g_{pq}$ are from the variation of the volume form. We now substitute (\ref{delta-g}). The first term in the second line becomes
\be
\left( - \frac{1}{8}\Phi^{ijkl} C_{ija} C_{kle}\Phi^{e}{}_{bcd}+ \frac{3}{8\cdot 28} \Phi^{ijkl} C_{ij}{}^p C_{klp} \Phi_{abcd} \right) \delta\Phi^{abcd} ,
\ee
while the second term is
\be
\frac{\kappa}{4}\left( C_{apq} C_e{}^{pq} \Phi^e{}_{bcd} - \frac{3}{28} (C_{ijk})^2 \Phi_{abcd}\right) \delta \Phi^{abcd}.
\ee
Thus, the variation of the action with respect to $\Phi^{abcd}$ is
\be\nonumber
E_{abcd}=\partial_{[a}C_{bcd]} - \frac{3}{2}  C_{[ab}{}^{p} C_{cd]p}  - \frac{1}{8}\Phi^{ijkl} C_{ij[a} C_{|kle|}\Phi^{e}{}_{bcd]}
+ \frac{3}{8\cdot 28} \Phi^{ijkl} C_{ij}{}^p C_{klp} \Phi_{abcd}\\ \nonumber 
+\frac{\kappa}{4}\left( C_{[a|pq} C_{e|}{}^{pq} \Phi^{e}{}_{bcd]} - \frac{3}{28} (C_{ijk})^2 \Phi_{abcd}\right) \\ \nonumber 
- \frac{1}{2\cdot 84} \Phi_{abcd} \left( \Phi^{ijkl} (\partial_i C_{jkl} -\frac{3}{2} C_{ij}{}^p C_{klp}) + \kappa (C_{ijk})^2 + \lambda \right).
\ee
This does not need to be zero, as the action also contains terms imposing the constraints guaranteeing that $\Phi_{abcd}$ is of the correct algebraic type. The constraint terms produce a variation that is an arbitrary tensor in $\Lambda^4_{27}$. So, we can only deduce that the $\Lambda^4_{35+1}$ and $\Lambda^4_{7}$ projection of the above vanishes. 
Before we extract these projections, it is worth evaluating the trace of the field equations. We have
\be
\Phi^{abcd} E_{abcd} = -2 \lambda - \frac{1}{2} \kappa (C_{abc})^2 - \Phi^{abcd} ( \partial_a C_{bcd}  - \frac{3}{4} C_{ab}{}^p C_{cdp}).
\ee
This is the projection of the field equations onto $\Lambda^4_1$, which must vanish. We therefore get the following consequence of the field equations
\be\label{feqs-trace}
\Phi^{abcd} \partial_a C_{bcd}  = \frac{3}{4} C_{ab}{}^p C_{cdp}-2\lambda- \frac{1}{2}\kappa (C_{abc})^2.
\ee
We can use this to simplify $E_{abcd}$. We have
\be\nonumber
E'_{abcd}=\partial_{[a}C_{bcd]} - \frac{3}{2}  C_{[ab}{}^{p} C_{cd]p}  - \frac{1}{8}\Phi^{ijkl} C_{ij[a} C_{|kle|}\Phi^{e}{}_{bcd]} +\frac{\kappa}{4} C_{[a|pq} C_{e|}{}^{pq} \Phi^{e}{}_{bcd]} \\ \label{E-prime}
+ \frac{1}{56} \Phi_{abcd} \Phi^{ijkl} C_{ij}{}^p C_{klp} + \frac{1 }{168} \Phi_{abcd}(\lambda-5 \kappa(C_{ijk})^2).
\ee
The $\Lambda^4_{35+1+7}$ projections of this vanish when the $\Lambda^4_{35+1+7}$ projections of $E_{abcd}$ vanish and vice versa, so $E'_{abcd}=0$ gives an equivalent encoding of field equations. 

\subsection{Extracting $\Lambda^4_{35+1+7}$ projections}

To understand the implications of the field equations we extract the $\Lambda^4_{35+1}$ and $\Lambda^4_{7}$ projections. This gives
\be\label{feqs}
\Phi_{b}{}^{pqr} E'_{apqr}  =\frac{1}{4} \Phi_{b}{}^{pqr} \nabla_{a}C_{pqr} - \frac{3}{4}  \Phi_{b}{}^{pqr}\nabla_r C_{apq}\\ \nonumber
- \frac{3}{2}\Phi_{b}{}^{pqr} C_{ap}{}^s C_{qrs} 
- \frac{3}{4}\Phi^{pqrs} C_{apq} C_{brs} + \frac{3\kappa}{2} C_a{}^{pq} C_{bpq}\\ \nonumber
+ \frac{1}{4} g_{ab}\left(\lambda -\frac{\kappa}{2} (C_{pqr})^2 + \frac{3}{4} \Phi^{pqrs} C_{pq}{}^p C_{rsp}\right).
\ee
Its $ab$ symmetrisation and anti-symmetrisation compute the $\Lambda^4_{35+1}$ and $\Lambda^4_{7}$ parts respectively. We wrote the derivatives here as the covariant derivatives, for the computations to follow. 

\subsection{Rewriting the $\kappa=0$ field equations - antisymmetric part}

For $\kappa=0$ we have $C=T$. Let us understand the arising field equations. We start with the anti-symmetric part. Taking (twice) the anti-symmetric part of the field equations (\ref{feqs}) we get
\be
\frac{1}{2} \Phi_{[a}{}^{pqr} \nabla_{b]}T_{pqr} - \frac{3}{2}  \Phi_{[a}{}^{pqr}\nabla_{|r|} T_{b]pq}- 3\Phi_{[a}{}^{pqr} T_{b]p}{}^s T_{qrs} =0.
\ee
With the help of the curvature identity (\ref{T-divergence-identity}) we can rewrite this as
\be
\nabla^r T_{r;ab} =0,
\ee
which is just vanishing of the divergence of the original torsion. This also makes it manifest that this equation is $\Lambda^2_7$ valued. Note also that this equation does not hold automatically. It is a non-trivial field equation to be imposed, and it becomes a second order PDE on the original 4-form. It can be interpreted as the evolution equation for the  $\Lambda^4_7$ part of the Cayley form perturbation, as is confirmed by the linearised analysis below. 

\subsection{Rewriting the $\kappa=0$ field equations - symmetric part}

For the analysis of the symmetric part, we take (twice) the symmetric part of (\ref{feqs}), also writing it with the opposite sign
\be\label{feq-symm}
-\frac{1}{2} \Phi_{(a}{}^{pqr} \nabla_{b)}T_{pqr} + \frac{3}{2}  \Phi_{(a}{}^{pqr}\nabla_{|r|} T_{b)pq}\\ \nonumber
+ 3 \Phi_{(a}{}^{pqr} T_{b)p}{}^s T_{qrs} 
+\frac{3}{2}\Phi^{pqrs} T_{apq} T_{brs} 
+ \frac{3}{8} g_{ab} \Phi^{ijkl} T_{ij}{}^p T_{klp} +\lambda g_{ab}=0.
\ee
Contract the resulting equation with $g^{ab}$ we get (\ref{feqs-trace}). 
Comparing this with (\ref{ricci-scalar}) we see that this is {\it not} the condition that the Ricci scalar is constant. Rather, using (\ref{ricci-scalar}), we can rewrite this equation as
\be
R = T_{abc} T^{abc} + \frac{1}{4} \Phi^{abcd} T_{ab}{}^p T_{cdp} + 4\lambda.
\ee
A computation shows that this can be rewritten as
\be
R= T^{abc} ( T+ \frac{1}{6} J_3(T))_{abc} = \frac{7}{6} (T_{abc}^{48})^2 + 4\lambda.
\ee
Here $T_{48} = \pi_{48}(T)$ is the $\Lambda^3_{48}$ part of the torsion 3-form. We thus see that the curvature scalar is sourced just by this part of the torsion. 

For the complete symmetric part of the equation, comparing this with (\ref{Ricci}), we can see that the second order part here does not reduce to that in $R_{ab}$. The comparison with (\ref{Ricci}) suggests that we can rewrite (\ref{feq-symm}) as
\be
3 R_{ab} + \Phi_{(a}{}^{pqr} \nabla_{b)}T_{pqr} \\ \nonumber
- 3 T_a{}^{pq} T_{bpq} +\frac{3}{2}\Phi^{pqrs} T_{apq} T_{brs}
+ \frac{3}{8} g_{ab} \Phi^{ijkl} T_{ij}{}^p T_{klp} + \lambda g_{ab}=0.
\ee
We thus see that the field equations {\it do not} state that the metric is Einstein. Instead, there are extra contributions coming from the torsion 3-form, and its derivatives. Note that the covariant derivative appears in this equation in such a way that, while both $R_{ab}$ and $\Phi_{(a}{}^{pqr} \nabla_{b)}T_{pqr}$ do depend on it, the specific combination of these terms that appears does not depend on $\nabla$. This will become more pronounced once we rewrite the field equations as a condition that a certain 4-form vanishes.

\subsection{Different ways of writing the field equations}

We note that we can introduce a symmetric tensor
\be
H_{ab}:= \Phi^{ijkl} T_{ija} T_{klb} - \frac{1}{7} g_{ab} \Phi^{ijkl} T_{ij}{}^p T_{klp}.
\ee
The 4-form encoding the field equations can then be written very compactly as
\be
E'_{abcd}=\partial_{[a}T_{bcd]} - \frac{3}{2}  T_{[ab}{}^{p} T_{cd]p}  - \frac{1}{8} H_{[a|e|} \Phi^e{}_{bcd]} + \frac{\lambda}{84} \Phi_{abcd}.
\ee
Recall that $E'_{abcd}$ is the tensor encoding the field equations of the theory, see (\ref{E-prime}). 
The field equations are then the statement that this equals to an arbitrary tensor in $\Lambda^4_{27}$, which we know can be parametrised as (\ref{Psi-Phi}). So, we get one of the possible ways of writing the field equations
\be
\partial_{[a}T_{bcd]} - \frac{3}{2}  T_{[ab}{}^{p} T_{cd]p}  - \frac{1}{8} H_{[a|e|} \Phi^e{}_{bcd]} + \frac{\lambda}{84} \Phi_{abcd} = \Psi_{[ab}{}^{pq} \Phi_{|pq| cd]}, 
\ee
where $\Psi^{abcd}$ is an arbitrary symmetric tracefree matrix in ${\rm Sym}_0^2(\Lambda^2_7)$.

\subsection{Yet another rewriting of the field equations}

Yet another way of writing the field equations, potentially useful, is obtained by 
computing $\Phi_{abc}{}^s \Phi_s{}^{pqr} E'_{dpqr}$, and anti-symmetrising on $abcd$. This gives a 4-form that is projected onto the $\Lambda^4_{35+1}$ and $\Lambda^4_7$ parts, eliminating the $\Lambda^4_{27}$ part of $E'_{abcd}$ that does not need to be zero. For a general 4-form we have
\be
\frac{1}{6}\Phi_{abc}{}^s \Phi_s{}^{pqr} \sigma_{dpqr} = (\mathbb{I}- \frac{1}{2} J_4)(\sigma)_{abcd},
\ee
explicitly showing that the $\Lambda^4_{27}$ component is projected away. We now apply this projector to the 4-form $E'_{abcd}$ to get the following 4-form field equations
\be\label{feqs-4-form}
\nabla_{[a} T_{bcd]} - \frac{3}{4} \Phi_{[ab}{}^{pq} \nabla_c T_{d]pq} - \frac{3}{4} \Phi_{[ab}{}^{pq} \nabla_{|p|} T_{cd]q} \\ \nonumber
-\frac{3}{2} T_{[ab}{}^p T_{cd]p} + \frac{3}{4} \Phi_{[ab}{}^{pq} T_{cd]}{}^{r} T_{pqr}- \frac{1}{8} \Phi_{[abc}{}^p \Phi^{ijkl} T_{d]ij} T_{klp} 
- \frac{3}{2} \Phi_{[ab}{}^{pq} T_{c|p|}{}^r T_{d]qr} \\ \nonumber
+ \frac{1}{32} \Phi_{abcd} \Phi^{ijkl}  T_{ij}{}^p T_{klp} + \frac{\lambda}{12} \Phi_{abcd}=0.
\ee
Since the first line here can be rewritten as
\be
(\mathbb{I} - \frac{1}{2} J_4)( \nabla_{[a} T_{bcd]}),
\ee
we see that the operator that appears in the field equations is built from the usual partial derivative, rather than the covariant one. 

\subsection{Analysis of the $\kappa=-2$ field equations}

In the general $\kappa$ case, we can rewrite the field equations (\ref{feqs}) in terms of the intrinsic torsion 3-form $T$, using the relation between $C$ and $T$. However, the arising general $\kappa$ results are too cumbersome. We now specialise to the particularly interesting case $\kappa=-2$, when the full non-linear equations of the theory reduce to Einstein equations, as we shall now see. 

We substitute $C$ in the form (\ref{C-GR}) to (\ref{feqs}) and whenever the derivatives get applied to the basic 4-form, evaluate them using (\ref{torsion-3-form}). The resulting field equations are as follows
\be\nonumber
\frac{1}{4} \Phi_{b}{}^{pqr} \nabla_{a}T_{pqr} - \frac{1}{2}  \Phi_{b}{}^{pqr}\nabla_r T_{apq}+  \frac{1}{4} \Phi_{a}{}^{pqr}\nabla_r T_{bpq}+ \frac{1}{2}( \nabla^p T_{abp} - \frac{1}{2} \Phi_{ab}{}^{cd} \nabla^p T_{cdp}) \\ \nonumber
 - \frac{1}{2} T_a{}^{pq} T_{bpq} - \frac{1}{4}\Phi_{(a}{}^{pqr} T_{b)p}{}^s T_{qrs} 
+\frac{1}{16}\Phi^{pqrs} T_{apq} T_{brs}- \frac{1}{16} \Phi^{pqrs} T_{abp} T_{qrs} + \frac{5}{8}\Phi_{[a}{}^{pqr} T_{b]p}{}^s T_{qrs}  \\ \nonumber
+ \frac{1}{32} \Phi_{ab}{}^{pq} \Phi^{ijkl} T_{pqi} T_{jkl} + \frac{1}{32} \Phi_{a}{}^{pqr} \Phi_b^{ijk} T_{pqi} T_{rjk} - \frac{1}{96} \Phi_{a}{}^{pqr} \Phi_b^{ijk} T_{pqr} T_{ijk}
\\ \nonumber
+ \frac{1}{4} g_{ab}\left(\lambda -\frac{1}{2} (T_{pqr})^2 - \frac{3}{4} \Phi^{pqrs} T_{pq}{}^p T_{rsp}+ \frac{1}{2} \Phi^{pqrs} \nabla_p T_{qrs} \right)=0.
\ee
To simplify this we first use (\ref{T-divergence}) in the last term in the first line, which gives
\be\nonumber
\frac{1}{4} \Phi_{(a}{}^{pqr} \nabla_{b)}T_{pqr} -  \frac{1}{4} \Phi_{(a}{}^{pqr}\nabla_{|r|} T_{b)pq} \\ \nonumber
 - \frac{1}{2} T_a{}^{pq} T_{bpq} - \frac{1}{4}\Phi_{(a}{}^{pqr} T_{b)p}{}^s T_{qrs} 
+\frac{1}{16}\Phi^{pqrs} T_{apq} T_{brs}\\ \nonumber
+ \frac{1}{32} \Phi_{a}{}^{pqr} \Phi_b^{ijk} T_{pqi} T_{rjk} - \frac{1}{96} \Phi_{a}{}^{pqr} \Phi_b^{ijk} T_{pqr} T_{ijk}\\ \nonumber
+ \frac{1}{4} g_{ab}\left(\lambda -\frac{1}{2} (T_{pqr})^2 - \frac{3}{4} \Phi^{pqrs} T_{pq}{}^p T_{rsp}+ \frac{1}{2} \Phi^{pqrs} \nabla_p T_{qrs} \right)=0
\ee
for the $ab$-symmetric part and 
\be\nonumber
 \frac{1}{32}\left( \Phi_{ab}{}^{pq} \Phi^{ijkl} T_{pqi} T_{jkl} - 2 \Phi^{pqrs} T_{abp} T_{qrs} - 12\Phi_{[a}{}^{pqr} T_{b]p}{}^s T_{qrs} \right) 
=0.
\ee
for the $ab$-anti-symmetric part. To simplify this quantity, we contract (\ref{identity}) with $T_{cdi} T^{jkl}$. This immediately shows that the quantity in brackets here vanishes. So, the anti-symmetric part of the field equations is satisfied identically. 

To simplify the symmetric part, we contract (\ref{identity}) with $T_{bcd} T^{jkl}$ to get the following identity
\be
\frac{1}{4} \Phi_{a}{}^{pqr} \Phi_b^{ijk} T_{pqi} T_{rjk}- \frac{1}{12} \Phi_{a}{}^{pqr} \Phi_b^{ijk} T_{pqr} T_{ijk}= \\ \nonumber
-2\Phi_{(a}{}^{pqr} T_{b)p}{}^s T_{qrs} + \frac{1}{2} g_{ab} \Phi^{ijkl} T_{ij}{}^p T_{klp} - \frac{1}{2} \Phi^{pqrs} T_{apq} T_{brs} .
\ee
Applying this to the third line of the symmetric part we get
\be\nonumber
\frac{1}{4} \Phi_{(a}{}^{pqr} \nabla_{b)}T_{pqr} -  \frac{1}{4} \Phi_{(a}{}^{pqr}\nabla_{|r|} T_{b)pq} 
 - \frac{1}{2} T_a{}^{pq} T_{bpq} - \frac{1}{2}\Phi_{(a}{}^{pqr} T_{b)p}{}^s T_{qrs} 
\\ \nonumber
+ \frac{1}{4} g_{ab}\left(\lambda -\frac{1}{2} (T_{pqr})^2 - \frac{1}{2} \Phi^{pqrs} T_{pq}{}^p T_{rsp}+ \frac{1}{2} \Phi^{pqrs} \nabla_p T_{qrs} \right)=0.
\ee
Using (\ref{Ricci}), (\ref{ricci-scalar}), and multiplying by $-2$, this can be rewritten in terms of the Ricci tensor as
\be
 R_{ab} + \frac{1}{4} g_{ab} (R- 2\lambda)=0 \qquad \Rightarrow \qquad \lambda = \frac{3R}{4}. 
\ee
Thus, the symmetric part of the field equations coincides with Einstein equations. Thus, the $\Lambda^3_{1+35}$ part of $\kappa=-2$ field equations are Einstein equations, and the $\Lambda^4_7$ part is satisfied identically. 

\subsection{Action rewritten in terms of torsion squared}

The obtained result that $\kappa=-2$ action leads to Einstein equations as its Euler-Lagrange equations is not surprising, given that the action in this case can be rewritten in terms of the Ricci scalar. Indeed, we start with the action (\ref{action-for-variation}), replace the derivative in the first term with the Levi-Civita covariant derivative, and integrate by parts to put the derivative on the Cayley form. We get
\be
S_\kappa[\Phi,C] =\frac{1}{3!}  \int \Big( -(\nabla_a \Phi^{abcd})   C_{bcd} - \frac{3}{2} \Phi^{abcd} C_{ab}{}^{p} C_{cdp}  \\ \nonumber
+ \kappa  C^{abc} C_{abc}  + \lambda  \Big) v_g d^8x.
\ee
Using (\ref{nabla-phi-contr}) we can write the action as
\be\nonumber
S_\kappa[\Phi,C] =\frac{1}{3!}  \int \Big( 2J_3(T)^{abc}   C_{abc} - J_3(C)^{abc} C_{abc} 
+ \kappa  C^{abc} C_{abc}  + \lambda  \Big) v_g d^8x.
\ee
Substituting (\ref{C-T}) we get for the torsion squared part
\be
\frac{1}{(6 - (5+\kappa)\kappa)^2} \Big( 2 (6 - (5+\kappa)\kappa) J_3(T)   (6T+\kappa J_3(T))  \\ \nonumber
- J_3(6T+\kappa J_3(T)) (6T+\kappa J_3(T))
+ \kappa  (6T+\kappa J_3(T)) (6T+\kappa J_3(T))\Big) \\ \nonumber
=\frac{1}{(6 - (5+\kappa)\kappa)^2} \Big(  (6\kappa T+(6 - 10\kappa -\kappa^2) J_3(T) - \kappa J^2_3(T)) (6T+\kappa J_3(T)) \Big).
\ee
Using $(J_3)^2 = 6 \mathbb{I} - 5 J_3$ and simplifying we get
\be
\frac{J_3(T)  (6T+\kappa J_3(T))}{6 - (5+\kappa)\kappa}    =\frac{(6-5\kappa) T J_3(T)  +6\kappa T^2}{6 - (5+\kappa)\kappa}   ,
\ee
and thus
\be\label{action-kappa-phi}
S_\kappa[\Phi] = \frac{1}{6}\int_M \left( \frac{(6-5\kappa) T J_3(T)  +6\kappa T^2}{6 - (5+\kappa)\kappa} + \lambda \right) v_g d^8 x. 
\ee
Thus, the general action is indeed the integral of a linear combination of $T^2$ and $T J_3(T)$. 

We now specialise to $\kappa=-2$ and get
\be\nonumber
S_{\kappa=-2}[\Phi] = \frac{1}{6} \int \Big(\frac{4}{3} T J_3(T)- T^2      + \lambda \Big) v_g d^8x \\ 
= - \frac{1}{6} \int \Big( R -  \lambda \Big)v_g d^8x,
\ee
where we used (\ref{S-EH}). Thus, the $\kappa=-2$ action, after the solution for $C$ is substituted back into the Lagrangian, is a multiple of the Einstein-Hilbert functional. This explains why the field equations in this case reproduce Einstein equations. Indeed, varying with respect to the metric we get
\be
R_{ab} - \frac{1}{2} g_{ab} (R -  \lambda )=0.
\ee
Taking the trace of this equation one gets $\lambda = 3R/4$, and thus this is the same equation

\section{Linearisation}
\label{sec:lin-action}

We now compute the linearisation of the general action (\ref{action-kappa-phi}) and verify that it gives the most general diffeomorphism-invariant linearised theory (\ref{L-diff}). 

We denote the linearisation of torsion $\delta T=t$. We have already computed the full non-linear action in the torsion squared form in (\ref{action-kappa-phi}). For our comparison with flat space linearised action we set $\lambda=0$. It is most useful to rewrite the linearisation of (\ref{action-kappa-phi}) in terms of $J_3(t)$ rather than $t$ itself, because there is a simpler explicit formula for $J_3(t)$, see below. This linearised action takes the form 
\[
S^{(2)}[\phi] =  \int {\mathcal L}^{(2)} =  
\frac{1}{6} \int \frac{\kappa J_3(t) J_3(t) +  J_3(t) (J_3+ 5\mathbb{I}) J_3(t)}{6 - (5+\kappa)\kappa} .
\]

We now use the parametrisation (\ref{phi-param}). In this parametrisation, using (\ref{epsilon-phi-3}) gives
\[
J_3(t)_{abc}= \frac{1}{48} \epsilon_{abc}{}^{pijkl} \partial_p \phi_{ijkl} = \\ \nonumber
-\frac{1}{2} \Phi_{abc}{}^p \partial^i ( h_{ip} - \frac{1}{4} \xi_{ip}) + \frac{1}{2}\Phi_{abc}{}^i \partial_i h - \frac{3}{2} \Phi_{[bc}{}^{ip} \partial_p ( h_{a]i} - \frac{1}{4} \xi_{a]i}). 
\]
An algebraic manipulation computation then gives
\be
\left(1- \frac{5\kappa}{6} - \frac{\kappa^2}{6}\right) {\mathcal L}^{(2)}=  \frac{1}{2} (1+\frac{\kappa}{6}) (\partial_a h_{bc})^2 -\frac{1}{6}(1-\frac{\kappa}{2}) (\partial_ a h)^2  \\ \nonumber 
 -\frac{1}{3} (1-\frac{\kappa}{2}) h \partial^a \partial^b h_{ab}- \frac{2}{3} (\partial^a h_{ab})^2 
+\frac{1}{24} (1+\frac{\kappa}{2})( \partial_a \xi_{bc})^2 - \frac{2}{3} (1+\frac{\kappa}{2}) \partial_b h^{ba} \partial^c \xi_{ca}.
\ee
This is the diffeomorphism-invariant Lagrangian of the type (\ref{L-xi}) with (\ref{parameters}) and 
\be
\rho = 1+\frac{\kappa}{6}, \quad \mu = \frac{2}{3} \left( 1+\frac{\kappa}{2}\right).
\ee
This shows that the linearisation of our general action gives the linearisation of the Einstein-Hilbert Lagrangian for $\kappa=-2$. We can also write down what the Lagrangian (\ref{elliptic-gauge}) becomes with this choice of the parameters. We get
\be
{\mathcal L}=\frac{\kappa+6}{12} ( \partial_a h_{bc})^2 + \frac{(\kappa+6)(\kappa-2)}{96} ( \partial_a h)^2 \\ \nonumber
+\frac{(\kappa+6)(\kappa+2)}{192}( \partial_a \xi_{bc})^2 
- \frac{2}{3} (\partial^a (h_{ab} - \frac{2-\kappa}{8} \eta_{ab} h + \frac{2+\kappa}{4} \xi_{ab}))^2.
\ee
This shows that another interesting point in the theory space is $\kappa=2$, when there is no separate kinetic term for the trace of the metric in the linearised Lagrangian. The value $\kappa=-6$ is also special. As is clear from (\ref{C-T}), for this value of $\kappa$ the tensor $C_{abc}$ can no longer be solved for in terms of $T_{abc}$. All these special cases need to be studied further to understand their significance.

\section*{Acknowledgements} I am grateful to my students Niren Bhoja and Adam Shaw for many discussions related to the material described here, and to Sam Close for spotting a variation mistake in the first version of this paper that led to incorrect description of the $\kappa=-2$ case. I am also grateful to Shubham Dwivedi for answering some questions and useful references.


\begin{thebibliography}{99}

\bibitem{Fernandez} M.\ Fern\'andez, A Classification of Riemannian Manifolds with Structure Group Spin(7), Ann. Mat. Pura  Appl. (IV) 143 (1986), 101-122.


\bibitem{Plebanski:1977zz}
J.~F.~Plebanski,
On the separation of Einsteinian substructures,
J. Math. Phys. \textbf{18} (1977), 2511-2520
doi:10.1063/1.523215

\bibitem{Bhoja:2024xbe}
N.~Bhoja and K.~Krasnov,
SU(2) structures in four dimensions and Plebanski formalism for GR,
[arXiv:2405.15408 [math.DG]].

\bibitem{Spiro-Spin7} Spiro Karigiannis, Flows of Spin(7)-structures, Proceedings of the 10th Conference on Differential Geometry and its Applications: DGA 2007; World Scientific Publishing, (2008), 263-277, also available as arXiv:0709.4594 [math.DG]. 

\bibitem{Moreno-Earp} D. Fadel, E. Loubeau, A. Moreno, and H.N. S\'a Earp, Flows of geometric structures, to appear in Crelle’s Journal,
arXiv:2211.05197 [math.DG] (2024).

\bibitem{Bryant} R.~Bryant, Some remarks on $G_2$-structures, Proceeding of Gokova Geometry-Topology Conference 2005 edited by S. Akbulut, T Onder, and R.J. Stern (2006), International Press, 75--109, also available as arXiv:math/0305124 [math.DG].

\bibitem{Spiro-G2} Shubham Dwivedi, Panagiotis Gianniotis, Spiro Karigiannis, Flows of G2-structures, II: Curvature, torsion, symbols, and functionals, arXiv:2311.05516 [math.DG].

\bibitem{Loubeau-Earp1} E. Loubeau and H.N. S\'a Earp, Harmonic flow of geometric structures, Annals of Global Analysis and Geometry 64,
arXiv:1907.06072 [math.DG] (2023).

\bibitem{Fowdar-Earp1} U. Fowdar and H.N. S\'a Earp, Flows of SU(2)-structures, arXiv:2407.03127 [math.DG] (2024).


\bibitem{Fowdar-Earp2} U. Fowdar and H.N. S\'a Earp, Harmonic flow of quaternion-K¨ahler structures, Journal of Geometric Analysis 34, 183,
arXiv:2301.12494 [math.DG] (2024).

\bibitem{Dwivedi-Earp} S. Dwivedi, E. Loubeau, and H.N. S\'a Earp, Harmonic flow of Spin(7)-structures, Annali della Scuola Normale Superiore di Pisa XXV (1), arXiv:2109.06340 [math.DG] (2024).

\bibitem{Loubeau} E. Loubeau, The analysis of the harmonic Spin(7) flow, in: Matem\'atica Contempor\"anea 60, Special Volume:
Proceedings of the Workshop on Geometric Structures and Moduli Spaces, UNC2022, arXiv:2311.17800 [math.DG]
(2024).


\bibitem{Dwivedi} S.~Dwivedi, A gradient flow of Spin(7)-structures, arXiv:2404.00870 [math.DG].

\bibitem{xact} xAct:  Efficient tensor computer algebra for the Wolfram Language, http://xact.es/index.html

\bibitem{Git} A Mathematica notebook containing all the definitions needed to check computations in this paper; some examples are also treated, https://github.com/kirikra3/spin7/blob/main/spin7.nb

\bibitem{Feger:2019tvk}
R.~Feger, T.~W.~Kephart and R.~J.~Saskowski,
LieART 2.0 \textendash{} A Mathematica application for Lie Algebras and Representation Theory,
Comput. Phys. Commun. \textbf{257} (2020), 107490
doi:10.1016/j.cpc.2020.107490
[arXiv:1912.10969 [hep-th]].

\bibitem{Ivanov} S.~Ivanov, Connection with torsion, parallel spinors and geometry of Spin(7) manifolds, Math. Res. Lett. 11 (2004), no. 2-3, 171--186.

\bibitem{Friedrich} Thomas Friedrich, On types of non-integrable geometries, arXiv:math/0205149 [math.DG].



\end{thebibliography}
\end{document}